\documentclass[11pt]{article}
\usepackage{float}

\usepackage{amsmath,amssymb,amsthm}
\usepackage{mathtools}
\usepackage{booktabs}
\usepackage{graphicx}
\usepackage{hyperref}
\usepackage{xcolor}
\usepackage[left=2cm, right=2cm, top=3cm, bottom=3cm]{geometry}

\newtheorem{theorem}{Theorem}
\newtheorem{lemma}[theorem]{Lemma}
\newtheorem{proposition}[theorem]{Proposition}
\newtheorem{remark}{Remark}
\newtheorem{definition}{Definition}

\newcommand{\Omg}{\Omega}
\newcommand{\bOmg}{\partial\Omega}
\newcommand{\C}{\mathcal{C}}
\newcommand{\Hone}{H^1(\Omega)}
\newcommand{\Hmone}{H^{-1}(\Omega)}
\newcommand{\Honehalf}{H^{1/2}(\partial\Omega)}

\newcommand{\norm}[1]{\left\|#1\right\|}
\newcommand{\seminorm}[1]{\left|#1\right|}
\newcommand{\mtil}{\widetilde{m}}
\newcommand{\Xint}{X}
\newcommand{\Zbnd}{Z}
\newcommand{\R}{\mathbb{R}}

\title{\Large\bf  Consistent CutPINNs for Elliptic PDEs on Curved Level-Set Domains}
\author{Maneesh Kumar Singh\thanks{Department of Mathematics, SRM Institute of
Science and Technology, Kattankulathur, India. Email:
\texttt{maneeshs@srmist.edu.in}}}

\date{\today}

\begin{document}
% =============================================================================

\maketitle

%\begin{abstract}
% Standard Physics-Informed Neural Networks (PINNs) penalise the interior
% PDE residual and boundary mismatch in $L^2$, a norm mismatch with the
% $H^1(\Omega)$ energy norm that governs well-posedness for second-order
% elliptic problems. The consistent PINN framework of \cite{bonito2025} closes this gap on flat domains via a discrete
% $H^{1/2}(\partial\Omega)$ surrogate built from the Kuhn--Tucker
% simplicial decomposition of $(0,1)^d$. We extend this framework to
% curved level-set domains $\Omega = \{\varphi < 0\} \subset \mathbb{R}^2$,
% where no such simplicial decomposition exists, via a \textit{Chord-arc} reduction
% that transfers the discrete boundary norm to a periodic interval with
% constants depending only on the maximum curvature. Combined with a
% cut-domain $L^2$ discretisation, this gives an \emph{a priori} bound on
% the $H^1$ error at the optimal recovery rate. Experiments on a disk and
% a non-convex flower confirm a $5$--$8$-fold error reduction and
% $26\times$ improvement in robustness to cut-cell configurations.
%\end{abstract}

\begin{abstract}
We propose \emph{Consistent CutPINN}, a framework for partial
differential equations posed on bounded curved domains defined
implicitly by a $\C^2$ level-set function, $\Omega = \{\varphi < 0\}$. In this paper we develop the framework for
second-order elliptic problems in two dimensions. The standard PINN
loss penalises the boundary mismatch in $L^2(\partial\Omega)$, but
$L^2(\partial\Omega)$ does not control the $H^{1/2}(\partial\Omega)$
trace norm that appears in the $H^1(\Omega)$ energy estimate. The
consistent PINN framework of Bonito et al.~\cite{bonito2025}
fixes this on the unit cube $(0,1)^d$ via a Kuhn--Tucker simplicial
decomposition of the flat boundary faces, but the construction relies
on the affine structure of the faces and does not carry over to
smooth curved boundaries. We address this gap. Specifically, (i) we
introduce a discrete $H^{1/2}(\partial\Omega)$ surrogate built
directly from collocation points on a $\C^2$ curve, (ii) we prove a
\textit{Chord-arc} norm equivalence between this surrogate and the continuous
trace norm, (iii) we establish an \emph{a priori} $H^1$ error
bound on cut domains, and (iv) we derive convergence rates under
Besov regularity using optimal recovery theory. Numerical experiments
on a disk and a non-convex flower domain confirm that the consistent
loss is much more accurate than the standard PINN loss and far more
robust to cut-cell configurations.
\end{abstract}


\section{Introduction}
\label{sec:intro}

Physics-informed neural networks (PINNs) were introduced by Raissi,
Perdikaris, and Karniadakis~\cite{raissi2019} as a way to solve
partial differential equations by training a neural network to
minimise a loss that penalises the PDE residual and the boundary
mismatch at scattered collocation points. Because the method is
meshfree, it is naturally suited to problems on curved domains,
where standard finite element or finite difference codes need
careful mesh generation and refinement near the boundary.

The theoretical literature on PINNs has grown steadily. Shin et
al.~\cite{shin2020convergence} proved convergence for second-order
elliptic and parabolic equations. Generalisation error estimates
were obtained in~\cite{mishra2023estimates, de2024numerical}, and
architecture-independent approximation bounds
in~\cite{de2022generic}. A broader treatment of neural network
approximation in the PDE context is~\cite{devore2021neural}.
Variants of the basic formulation have been studied extensively:
domain-decomposition extensions~\cite{hu2022extended}, variational
reformulations~\cite{anandh2025fastvpinns, anandh2025improving},
stabilised schemes for convection-dominated
problems~\cite{cengizci2026pinn, raina2026application}, hybrid
time-stepping for phase-field
equations~\cite{singh2024ss, singh2026mc}, and adaptive
collocation~\cite{visser2026pacmann}. Despite this work, direct
comparisons with classical finite element
solvers~\cite{grossmann2023pinn} show that standard PINNs are often
less accurate per degree of freedom, in particular when the
solution has boundary layers or high-frequency oscillation.

One reason for this underperformance is a norm mismatch in the
standard PINN loss, identified and analysed by Bonito et
al.~\cite{bonito2025}. The standard loss penalises the interior
residual in $L^2(\Omg)$ and the boundary mismatch in $L^2(\bOmg)$.
However, well-posedness of a second-order elliptic problem is
governed by the $H^1(\Omg)$ energy norm. The relevant \emph{a priori}
estimate controls $\norm{u-v}_{H^1(\Omg)}$ by the interior residual
in $H^{-1}(\Omg)$ and the boundary residual in $H^{1/2}(\bOmg)$. The
embedding $H^{1/2}(\bOmg)\subset L^2(\bOmg)$ is strict, so the
$L^2(\bOmg)$ penalty does not control the $H^{1/2}(\bOmg)$ trace
norm, and the standard loss is in this sense \emph{inconsistent}:
driving it to zero gives no guarantee on the $H^1(\Omg)$ error. A
standard counterexample is a high-frequency oscillation that
vanishes at the boundary collocation points but has arbitrarily
large $H^{1/2}(\bOmg)$ seminorm. The connection between this
phenomenon and optimal recovery under Besov regularity has been
explored in~\cite{devore1993besov, cohen2022optimal}, and related
inconsistency issues are discussed
in~\cite{zeinhofer2025unified, frerichs2026loss}.

Bonito et al.~\cite{bonito2025} addressed this by designing
\emph{consistent} loss functions. The boundary penalty is replaced
by a discrete surrogate of the $H^{1/2}(\bOmg)$ norm built from a
double-sum discretisation of the intrinsic seminorm, and the
interior $L^2(\Omg)$
term by a discrete $H^{-1}(\Omg)$ surrogate. The resulting loss is
norm-equivalent to the $H^1(\Omg)$ error through the classical
energy estimate, and the authors prove optimal convergence rates
under Besov model class assumptions. The framework has since been
used for higher-order elliptic PDEs~\cite{mishra2026consistent} and
for variational inequalities~\cite{khan2026mixed}. There is one
structural restriction common to all of this work: the discrete
$H^{1/2}(\bOmg)$ norm of~\cite{bonito2025} is built from the
Kuhn--Tucker simplicial decomposition of the flat boundary faces of
the unit cube $(0,1)^d$, and the proof uses the affine structure of
each face. The analysis does not transfer to boundaries that are
smooth closed curves rather than unions of flat segments.

A related but methodologically different line of work is the
unfitted finite element interpolated neural network
($\ell^2$-FEINN) framework of Li et al.~\cite{li2026}, which
combines a neural network parameterisation with aggregated CutFEM
test spaces~\cite{badia2018aggregated}. Unlike our setting, this
approach uses an underlying finite element mesh and sits outside
the optimal-recovery framework of~\cite{bonito2025} that we adopt,
so we do not include a direct comparison. A broader survey of
PINN approaches to complex geometries can be found
in~\cite{plankovskyy2025review}.

No existing framework combines consistent loss functions with
implicitly defined curved domains. We construct one. The setting
is a bounded two-dimensional domain $\Omg = \{\varphi < 0\}$
defined by a $\C^2$ level-set function $\varphi$, and we build the
discrete $\Honehalf$ norm directly from collocation points on the
curve, using the ambient Euclidean kernel
$|g(z)-g(z')|^2/|z-z'|^d$, with no underlying mesh or finite
element structure. We call the resulting method the
\emph{Consistent CutPINN framework}.

Let $\Omg \subset \R^2$ be a bounded domain given by a level-set
function $\varphi\colon \R^2 \to \R$:
\[
  \Omg = \bigl\{x \in \R^2 : \varphi(x) < 0\bigr\},
  \qquad
  \bOmg = \bigl\{x \in \R^2 : \varphi(x) = 0\bigr\}.
\]
Consider the Poisson problem with Dirichlet data,
\begin{equation}\label{eq:poisson}
  -\Delta u = f \quad \text{in } \Omg,
  \qquad
  u = g \quad \text{on } \bOmg,
\end{equation}
with $f \in \Hmone$ and $g \in \Honehalf$. Under standard
assumptions on $\Omg$ the problem has a unique weak solution
$u \in \Hone$, and for all $v \in \Hone$ one
has the two-sided bound
\begin{equation}\label{eq:energy_equiv}
  \norm{u - v}_{\Hone}
  \;\asymp\;
  \norm{f + \Delta v}_{\Hmone}
  + \norm{g - v}_{\Honehalf},
\end{equation}
see~\cite{bonito2025}. A consistent loss is one that approximates
the right-hand side of~\eqref{eq:energy_equiv}: if a computable
surrogate of it can be driven to zero during training, then the
$H^1$ error on the left must also go to zero.

The main theoretical result is the \emph{a priori} bound of
Theorem~\ref{thm:apriori}: minimising $L^*_{\mathrm{sq},2}$
controls the $H^1(\Omg)$ error at the optimal recovery rate, with
constants depending on the maximum curvature of $\bOmg$ through
the \textit{Chord-arc} constant. The key technical ingredient behind this
bound is Theorem~\ref{thm:h12_curve}, a two-sided equivalence
between the discrete and continuous $H^{1/2}(\bOmg)$ norms on a
$\C^2$ curve for arc-length-equidistant collocation points, with an
explicit recovery rate $m^{-\beta}$ set by the Besov smoothness of
the boundary datum.

We test the framework numerically on two level-set domains, a disk
and a non-convex flower with five-fold symmetry. The experiments
cover three aspects. Convergence is studied as the interior budget
$\mtil$ grows on the disk (Section~\ref{sec:exp1}) and on the
flower (Section~\ref{sec:exp4}). The corresponding training
histories at a fixed budget (Section~\ref{sec:exp2}) show whether
the loss tracks the $H^1$ error. Robustness is studied by moving
the disk centre through 121 positions along the bounding-box
anti-diagonal (Section~\ref{sec:exp3}), following~\cite{li2026},
and by replacing the equally-spaced boundary sampling with
{\em i.i.d} Monte Carlo sampling (Section~\ref{sec:exp6}). Finally,
the spatial distribution of the pointwise error on both domains is
shown in Section~\ref{sec:exp5}.
% In all settings, $L^*_{\mathrm{sq},2}$ achieves at least $5\times$
% lower $H^1$ error than the standard loss $L_{\mathrm{sq},1}$.

The paper is organised as follows. Section~\ref{sec:formulation}
fixes notation for the function spaces and sets up the PINN
pipeline on a level-set domain. Section~\ref{sec:losses} writes
down the four loss functions and defines the discrete $H^{1/2}$
norm on a curved boundary. The \textit{Chord-arc} lemma, the
norm-equivalence theorem, and the \emph{a priori} error bound are
proved in Section~\ref{sec:theory}. Section~\ref{sec:numerics}
reports the numerical experiments, and
Section~\ref{sec:conclusion} discusses limitations and future
work.

\section{Problem Formulation}
\label{sec:formulation}

This section collects the analytic and computational ingredients
used in the rest of the paper. Section~\ref{sec:functional} fixes
notation for the Sobolev and Besov spaces. We next set up the PINN
training pipeline, namely the collocation data, the trial network,
and the automatic differentiation of the Laplacian. The last
subsection describes the geometric setup of $\C^2$ level-set
boundaries that motivates the consistent loss design of
Section~\ref{sec:losses}.

\subsection{Function spaces}\label{sec:functional}
 
We recall the function spaces used throughout.
For a bounded Lipschitz domain $\Omg\subset\R^2$,
the Sobolev space $H^1(\Omg)$ consists of $L^2$ functions
with square-integrable first derivatives, equipped with the norm
$\norm{v}_{H^1(\Omg)}^2 := \norm{v}_{L^2(\Omg)}^2
+ \norm{\nabla v}_{L^2(\Omg)}^2$.
The dual space $H^{-1}(\Omg) := (H^1_0(\Omg))^*$ satisfies
the continuous embedding
\begin{equation}\label{eq:L2_embed}
  \norm{w}_{H^{-1}(\Omg)} \le C_P\,\norm{w}_{L^2(\Omg)},
  \qquad w \in L^2(\Omg),
\end{equation}
where $C_P$ is the Poincar\'e constant of $\Omg$. 
The \textbf{fractional Sobolev space} $H^{1/2}(\bOmg)$ is defined
via the semi-norm
\begin{equation}\label{eq:gagliardo_cont}
  |g|^2_{H^{1/2}(\bOmg)}
  := \int_{\bOmg}\int_{\bOmg}
     \frac{|g(z) - g(z')|^2}{|z - z'|^d}\,
     d\sigma(z)\,d\sigma(z'),
\end{equation}
with full norm
$\norm{g}^2_{H^{1/2}(\bOmg)}
:= \norm{g}^2_{L^2(\bOmg)} + |g|^2_{H^{1/2}(\bOmg)}$.
This is the natural trace space for $H^1(\Omg)$:
any $v\in H^1(\Omg)$ satisfies
$\norm{v|_{\bOmg}}_{H^{1/2}(\bOmg)}
\le C_{\mathrm{tr}}\,\norm{v}_{H^1(\Omg)}$.
For $s > 0$ and $1 \le p \le \infty$, the \textbf{Besov space}
$B^s_p(\Omg)$ consists of functions in $L^p(\Omg)$ with
finite seminorm
\begin{equation}\label{eq:besov_def}
  |f|_{B^s_p(\Omg)}
  := \sup_{k \ge 0}\; 2^{ks}\,
     \omega_r(f, 2^{-k})_{L^p(\Omg)},
\end{equation}
where $r > s$ is any integer and
$\omega_r(f,t)_{L^p(\Omg)}
:= \sup_{|h|\le t}
\norm{\Delta^r_h f}_{L^p(\Omg_{rh})}$
is the $r$-th modulus of smoothness with
$\Omg_{rh} := \{x\in\Omg : [x, x+rh]\subset\Omg\}$.
The norm is
$\norm{f}_{B^s_p(\Omg)} :=
\norm{f}_{L^p(\Omg)} + |f|_{B^s_p(\Omg)}$.
We write $U(B^s_p(\Omg))$ for the unit ball of $B^s_p(\Omg)$.
When $s > 2/p$ (and $d=2$), the embedding
$B^s_p(\Omg) \hookrightarrow \C(\Omg)$ holds,
so point evaluation is well-defined for functions in $B^s_p$.
We use the standard notation $t_+ := \max(t, 0)$ for the
positive part of a real number $t$.
For a detailed treatment of Besov spaces on domains,
see~\cite{devore1993besov} and references therein.

Throughout this paper, $C$ denotes a generic
positive constant that may depend on $r, \bar s, s, p, \kappa_{\max},
L, |\Omega|$, and the Stein extension constant $C_E$, but never on
$m, \tilde m$, the solution $u$, or the trial function $v$. Its value
may change from line to line. We write $a \lesssim b$ to mean
$a \le C b$ for such a constant $C$, and $a \asymp b$ to mean both
$a \lesssim b$ and $b \lesssim a$.

\subsection{Collocation data}

In the PINN setting the data $f$ and $g$ are never used in full.
The algorithm sees only their values at a finite set of points.
The method is meshfree in this sense: no quadrature rule or finite
element assembly is required, and the same code applies to any
domain for which a level-set function is available. Approximation
quality then depends entirely on how well the collocation points
sample the domain and its boundary, which is the question studied
classically in approximation theory and optimal
recovery~\cite{cohen2022optimal, devore1993besov}.

The method accesses $\mathbf{f} = (f_1,f_2,\ldots,f_{\mtil})$ and
$\mathbf{g} = (g_1,g_2,\ldots,g_{m})$ only at the
\emph{collocation sites}
\[
  \Xint = \{x_1,\ldots,x_{\mtil}\} \subset \Omg,
  \qquad
  \Zbnd = \{z_1,\ldots,z_m\} \subset \bOmg.
\]
Interior points are produced by rejection sampling from the bounding
box $[0,1]^2$, retaining only those satisfying $\varphi(x)<0$.
This is the standard strategy for implicitly defined domains and
requires no mesh. The only input is the level-set function,
cf.~\cite{li2026}. Boundary points are placed parametrically
on $\bOmg$: uniformly in arc-length for the disk, and via the exact
polar parametrisation for the flower domain. Parametric boundary
sampling ensures that points are equally spaced in arc-length, which
is the hypothesis required by the norm-equivalence
Theorem~\ref{thm:h12_curve}.

The interior and boundary collocation counts are linked by
$m = \lfloor\sqrt{\mtil}\rfloor$, giving the balanced scaling
$m \asymp \mtil^{1/(d-1)}$, as recommended
in~\cite{bonito2025}. For the four values $\mtil \in \{100, 400,
900, 1600\}$ used in the convergence experiments, the corresponding
boundary counts are $m \in \{10, 20, 30, 40\}$. Figure~\ref{fig:pipeline}
illustrates the collocation setup and the full network training pipeline.

% %% ── FIGURE 1: domain + pipeline 
\begin{figure}[htb!]
\centering
\includegraphics[width=\textwidth]{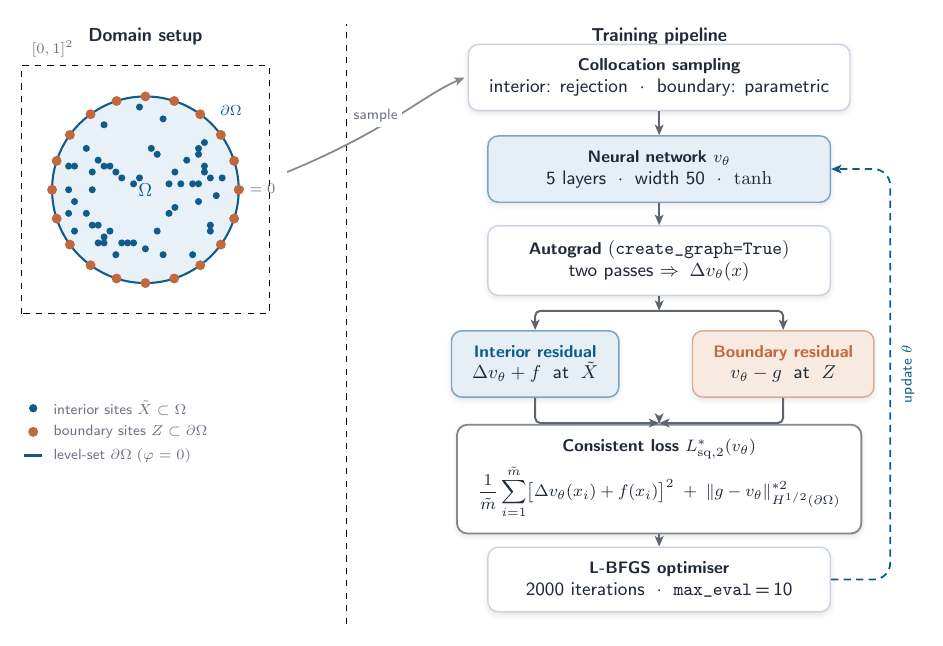}
\caption{Left: collocation setup on the cut domain
$\Omega=\{\varphi<0\}$.
Interior sites $\tilde{X}$ (blue, 55 shown) are drawn by rejection
sampling; boundary sites $Z$ (red, 20 shown) are placed parametrically
on $\partial\Omega$.
Right: training pipeline from collocation sampling through the neural
network $v_\theta$, automatic differentiation of the Laplacian
$\Delta v_\theta$, assembly of the consistent loss
$L^*_{\mathrm{sq},2}$, and L-BFGS optimisation.
The dashed arrow denotes the parameter update feedback loop.}
\label{fig:pipeline}
\end{figure}

The trial function $v_\theta\colon\R^2\to\R$ is a fully-connected
feedforward network with $L=5$ hidden layers, width $W=50$, and
$\tanh$ activations, giving $10{,}401$ trainable parameters. It takes
the two-dimensional coordinate $(x_1,x_2)$ as input and returns a
scalar. Parameters are initialised by the default PyTorch scheme
(Kaiming uniform). The network is trained by minimising one of
the loss functions defined in Section~\ref{sec:losses}.

\paragraph{Implementation details.} In our setup, 
at each L-BFGS step, $v_\theta$ is evaluated at the collocation sites
$\tilde{X}$ and $Z$. The Laplacian $\Delta v_\theta(x_i)$ is computed
via two backward passes of automatic differentiation with
\texttt{create\_graph=True}: the first pass gives $\nabla v_\theta$
and retains the computational graph, the second differentiates
$\nabla v_\theta$ to recover the second derivatives. The discrete
$H^{1/2}(\partial\Omega)$ double sum~\eqref{eq:gagliardo_cont} is
assembled in vectorised form, with cost $O(m^2)$ per evaluation,
dominated for $m \le 40$ by the forward and backward passes through
$v_\theta$. The \texttt{torch.optim.LBFGS} optimiser uses a history
of 10 past gradient updates with strong Wolfe line search. The
manufactured-solution data $(f, g)$ are precomputed once and stored
as torch tensors, so each training step is a pure forward-backward
pass with no host-side computation.

\subsection{Extension to curved boundaries}

Existing consistent PINN constructions provide discrete surrogates
$L^*$ of the right-hand side of~\eqref{eq:energy_equiv} that
preserve the norm equivalence up to a discretisation error. The
construction sits on the Kuhn--Tucker simplicial decomposition of
the flat boundary faces of $(0,1)^d$, summing the intrinsic-seminorm
kernel $|g(z) - g(z')|^2 / |z - z'|^d$ over the grid points on
these faces. The equivalence proof relies on the affine structure
of each face.

On a cut domain $\partial\Omega$ is a smooth closed curve and no
natural simplicial decomposition is available. Whether the same
kernel, evaluated on scattered points on the curve, still gives a
norm controlling the continuous $H^{1/2}$ trace is not obvious. On
a $C^2$ curve, however, the Euclidean chord distance is comparable
to the arc-length distance, with a positive constant $c_\Gamma$
depending on the curve. This chord-arc comparison is established in
Lemma~\ref{lem:chord_arc}.

Throughout the paper we assume the level-set function $\varphi$ is
$C^2$ with $|\nabla\varphi| > 0$ on $\partial\Omega$. The implicit
function theorem then guarantees that $\partial\Omega$ is a simple
closed $C^2$ curve with finite maximum curvature $\kappa_{\max}$. No
finite element mesh or boundary decomposition is required.

The experiments in Section~\ref{sec:numerics} confirm that the
consistent losses built this way produces $H^1$ errors five to eight
times smaller than the standard PINN loss across the collocation
budgets and geometries we test.

\section{Consistent Loss Functions}
\label{sec:losses}

The energy estimate~\eqref{eq:energy_equiv} suggests how to design
a loss whose minimisation controls the $H^1(\Omg)$ error: replace
the continuous norms $\norm{\cdot}_{H^{-1}(\Omg)}$ and
$\norm{\cdot}_{\Honehalf}$ on its right-hand side by computable
discrete surrogates that retain the norm equivalence up to a
discretisation error. We follow this strategy on a cut domain.
The interior term carries over from the flat-domain case, while
the boundary term needs a new discrete $\Honehalf$ norm built
from collocation points on a curved level-set boundary.

\subsection{Standard PINN losses}

We first record the standard PINN loss, which serves as the
baseline throughout the experiments. It does not try to match the
analytic structure of the energy estimate~\eqref{eq:energy_equiv}. 
The interior residual and the boundary mismatch are treated as two
independent least-squares terms, both in $L^2$. The formulation is
simple and is used in most PINN
implementations~\cite{raissi2019, cuomo2022}. The theoretical
shortcoming, namely that the $L^2(\bOmg)$ penalty does
not control the $\Honehalf$ trace norm, is what the consistent
losses below fix.

The original loss~\cite{raissi2019} with equal weighting is
\begin{equation}\label{eq:Lsq1}
  L_{\mathrm{sq},1}(v)
  = \frac{1}{\mtil}\sum_{i=1}^{\mtil}
    \bigl[\Delta v(x_i) + f(x_i)\bigr]^2
  + \frac{1}{m}\sum_{j=1}^{m}
    \bigl[v(z_j) - g(z_j)\bigr]^2.
\end{equation}
A weighted variant with $\lambda(m) = m^{1/(d-1)}$ is also
considered\footnote{Our convention places the factor $1/m$
explicitly on the boundary sum, in contrast to the formulation
in~\cite{bonito2025} eq.~(8.5) which absorbs it into the weight.
The two conventions coincide for $d=2$, the only case treated here;
for $d\geq 3$ they differ by an overall factor of $m$.}:
\begin{equation}\label{eq:Lsqlam}
  L_{\mathrm{sq},\lambda}(v)
  = \frac{1}{\mtil}\sum_{i=1}^{\mtil}
    \bigl[\Delta v(x_i) + f(x_i)\bigr]^2
  + m^{(2-d)/(d-1)}
    \frac{1}{m}\sum_{j=1}^{m}
    \bigl[v(z_j) - g(z_j)\bigr]^2.
\end{equation}
In two dimensions the exponent $(2-d)/(d-1)$ vanishes, so
$L_{\mathrm{sq},\lambda} \equiv L_{\mathrm{sq},1}$ identically. We
include both in the experiments as a sanity check, and indeed they
agree to machine precision in every run.

\subsection{Discrete $H^{1/2}$ norm on a curved boundary}

The \textit{intrinsic semi-norm} records how much a function
oscillates between pairs of boundary points, with the variation
weighted by the inverse-square distance. On flat domains the
integral is discretised over a structured grid on the boundary, and
the simplicial geometry of the flat faces is what gives the norm
equivalence. On a curved boundary no such grid exists. The natural
question is whether the same kernel, evaluated on collocation
points $\Zbnd = \{z_1,\ldots,z_m\}$ on the curve with the ambient
Euclidean distance, still controls the continuous $\Honehalf$
trace norm. This is the question of Section~\ref{sec:theory}, and
the answer is yes for $\C^2$ level-set boundaries in $d=2$. The
semi-norm on $\bOmg$ is
\begin{equation}\label{eq:h12_cont}
  \seminorm{g}^2_{H^{1/2}(\bOmg)}
  = \int_{\bOmg}\int_{\bOmg}
    \frac{|g(z)-g(z')|^2}{|z-z'|^d}\,dz\,dz'.
\end{equation}
We discretise this integral by replacing it with a double sum over
the collocation points $\Zbnd$ on the curve.

\begin{definition}[\textbf{Discrete $H^{1/2}$ norm on $\bOmg$}]
\label{def:h12}
Write $r_j = g(z_j) - v(z_j)$ for the boundary residual. Set
\begin{align}
  \norm{g-v}^{*2}_{L^2(\bOmg)}
  &= \frac{1}{m}\sum_{j=1}^{m} r_j^2,
  \label{eq:h12_l2}\\[4pt]
  \seminorm{g-v}^{*2}_{H^{1/2}(\bOmg)}
  &= \frac{1}{m^2}\sum_{\substack{i,j=1\\i\neq j}}^{m}
    \frac{(r_i - r_j)^2}{|z_i - z_j|^d},
  \label{eq:h12_semi}\\[4pt]
  \norm{g-v}^{*2}_{H^{1/2}(\bOmg)}
  &= \norm{g-v}^{*2}_{L^2(\bOmg)}
  + \seminorm{g-v}^{*2}_{H^{1/2}(\bOmg)}.
  \label{eq:h12_full}
\end{align}
\end{definition}

\begin{remark}\label{rem:equiv}
On the flat faces of $(0,1)^d$ the equivalence between the discrete
norm~\eqref{eq:h12_full} and the true $\Honehalf$ norm is proved
in~\cite{bonito2025} (Lemma~6.2 and Theorem~6.3) using the
Kuhn-Tucker simplicial structure of $\bOmg$. In the next
section,  we will extend this equivalence to smooth
curved boundaries, using a \textit{Chord-arc} argument for $d=2$.
\end{remark}

\subsection{Consistent losses}

We define two consistent losses, differing only in how the interior
residual is measured. The first uses an $L^\gamma$ interior penalty
with $\gamma = 1 + 1/\log\mtil$. This choice is motivated by the
embedding $L^\gamma(\Omg) \hookrightarrow H^{-1}(\Omg)$, which holds
for $\gamma \geq 2d/(d+2)$. In $d=2$ this gives $\gamma \geq 1$, and
$\gamma = 1 + 1/\log\mtil$ approaches this lower bound as the
collocation budget grows. The second variant retains the $L^2$
interior penalty of the standard PINN loss and replaces only the
boundary term by the discrete $\Honehalf$ norm.

The first consistent loss reads
\begin{equation}\label{eq:Lcg}
  L^*_{\mathrm{sq},\gamma}(v)
  = \left(\frac{1}{\mtil}\sum_{i=1}^{\mtil}
    \bigl|\Delta v(x_i)+f(x_i)\bigr|^\gamma\right)^{2/\gamma}
  + \norm{g-v}^{*2}_{H^{1/2}(\bOmg)},
\end{equation}
which expands explicitly as
\begin{align}
  L^*_{\mathrm{sq},\gamma}(v)
  :={}&
  \left[
    \frac{1}{\mtil}\sum_{i=1}^{\mtil}
    \bigl|f(x_i) + \Delta v(x_i)\bigr|^{\gamma}
  \right]^{2/\gamma}
  +
  \frac{1}{m^2}
  \sum_{\substack{i,j=1\\i\neq j}}^{m}
  \frac{\bigl|[g-v](z_i)-[g-v](z_j)\bigr|^2}
       {|z_i-z_j|^{2}}
  \nonumber\\[1ex]
  &+\;
  \frac{1}{m}\sum_{j=1}^{m}
  \bigl|g(z_j)-v(z_j)\bigr|^{2}.
\end{align}
The second consistent loss keeps the $L^2$ interior term:
\begin{equation}\label{eq:Lc2}
  L^*_{\mathrm{sq},2}(v)
  = \frac{1}{\mtil}\sum_{i=1}^{\mtil}
    \bigl[\Delta v(x_i)+f(x_i)\bigr]^2
  + \norm{g-v}^{*2}_{H^{1/2}(\bOmg)}.
\end{equation}
The consistent loss $L^*_{\mathrm{sq},2}$ modifies the standard loss
$L_{\mathrm{sq},1}$ only in the boundary term, replacing the
$L^2(\bOmg)$ penalty by the discrete $\Honehalf$ norm. This is the
minimal modification that makes the loss consistent. The two
consistent losses share this boundary term and differ only in the
interior penalty, which is $L^\gamma$ in $L^*_{\mathrm{sq},\gamma}$
and $L^2$ in $L^*_{\mathrm{sq},2}$.

%% ── FIGURE 2: network architecture 
\begin{figure}[!t]
  \centering
  \includegraphics[width=\textwidth]{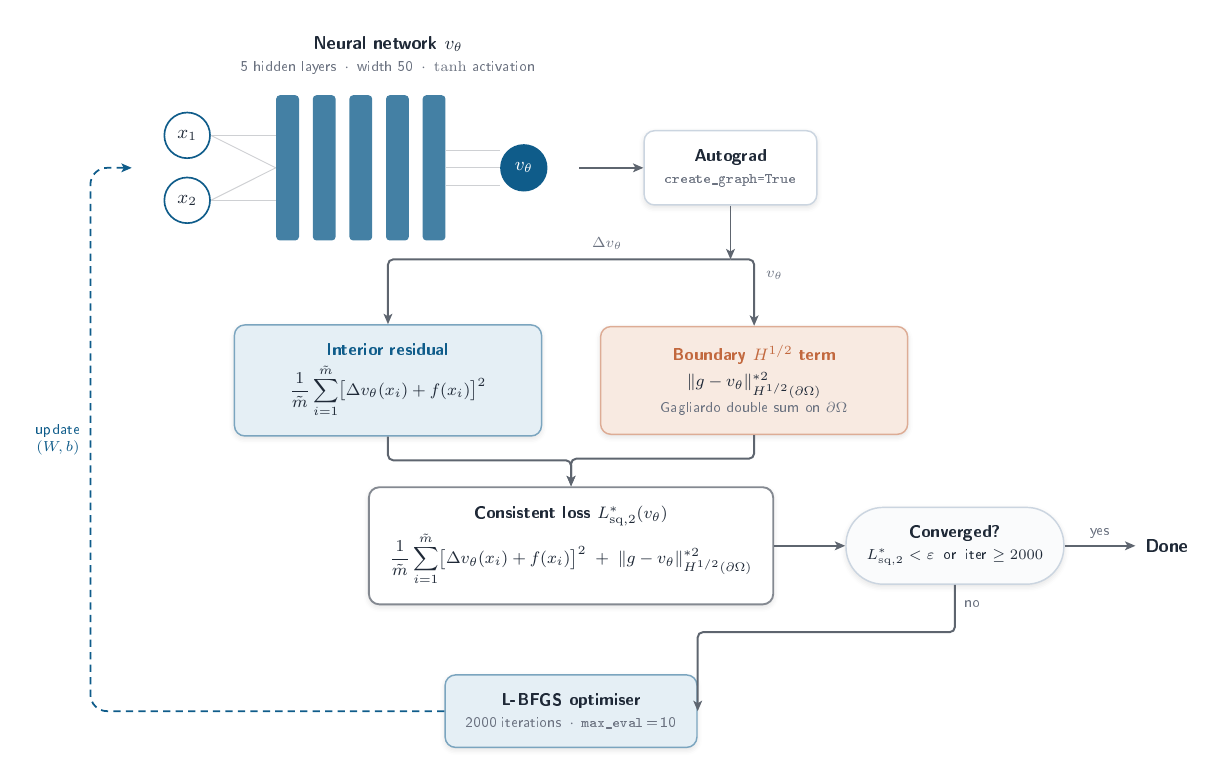}
  \caption{Decomposition of the consistent loss $L^*_{\mathrm{sq},2}$
  defined in~\eqref{eq:Lc2}. The network output $v_\theta$ and its
  Laplacian $\Delta v_\theta$ feed two residual terms: the
  interior $L^2$ residual at $\tilde X$ and the boundary $H^{1/2}$
  residual at $Z$, which combine into the scalar loss. Training
  stops when $L^*_{\mathrm{sq},2}$ falls below a tolerance
  $\varepsilon$ or 2000 L-BFGS iterations are reached.}
 \label{fig:network}
\end{figure}

\subsection{Optimal recovery framework}\label{sec:OR}
 
The consistent loss is designed to achieve the best possible
approximation rate given the available point samples. We make this
precise via the optimal recovery framework.
 
Assume the input data are constrained to the unit balls
\begin{equation}\label{eq:OR_classes}
  f \in \mathcal{F} := U(B^s_p(\Omg)),
  \quad s > 2/p,\;1\le p\le\infty,
  \qquad
  g \in \mathcal{G} := \mathrm{Tr}\bigl(
    U(B^{\bar{s}}_\infty(L^2(\Omg)))\bigr),
  \quad \bar{s} > 1.
\end{equation}
Since $\mathcal{F}$ is compactly embedded in $\C(\Omg)$
and $\mathcal{G}$ embeds into $\C(\bOmg)$,
the solution $u$ of~\eqref{eq:poisson} lies within a
compact subset of $H^1(\Omg)$:
\[
  \mathcal{U} := \bigl\{u \in \C(\Omg) \;\big|\;
    u \text{ solves \eqref{eq:poisson} for }
    f\in\mathcal{F},\; g\in\mathcal{G}\bigr\}.
\]
 
For sampled values
$\hat{f} = (\hat{f}_1,\ldots,\hat{f}_{\mtil})$ and
$\hat{g} = (\hat{g}_1,\ldots,\hat{g}_m)$,
the accessible knowledge of $f$ and $g$ is collected as
\[
  \mathcal{F}_{\mathrm{data}}
  := \{f\in\mathcal{F} : f(x_i) = \hat{f}_i,\;
       i=1,\ldots,\mtil\},
  \qquad
  \mathcal{G}_{\mathrm{data}}
  := \{g\in\mathcal{G} : g(z_j) = \hat{g}_j,\;
       j=1,\ldots,m\}.
\]
The solution is then constrained to the admissible set
\[
  \mathcal{U}_{\mathrm{data}}
  := \bigl\{u \in \mathcal{U}
    \;\big|\;
    -\Delta u(x_i) = \hat{f}_i,\;
    u(z_j) = \hat{g}_j,\;
    i=1,\ldots,\mtil,\;
    j=1,\ldots,m
  \bigr\}.
\]
 
The uniform optimal recovery rate over all sampling
configurations of size $|\tilde{X}| + |Z| = \mtil + m$ is
\[
  R^*_{\mtil,m}(\mathcal{U})_{H^1(\Omg)}
  := \inf_{\substack{\tilde{X}\subset\Omg,\;
     Z\subset\bOmg \\ |\tilde{X}|=\mtil,\;|Z|=m}}
     \;\sup_{u\in\mathcal{U}}\;
     R(\mathcal{U}_{\mathrm{data}}(u))_{H^1(\Omg)}.
\]
For the model classes $\mathcal{F}=U(B^s_p(\Omg))$ and
$\mathcal{G}=\mathrm{Tr}(U(B^{\bar{s}}_\infty(L^2(\Omg))))$ in $d=2$,
\cite[Theorems~3.1 and~3.2]{bonito2025} give the unrestricted
optimal recovery rate
$R^*_{\mtil,m}(\mathcal{U})_{H^1(\Omg)}\asymp\mtil^{-\alpha_{\rm OR}}
+m^{-\beta}$ with $\alpha_{\rm OR}=s/2$ for $p>1$ and
$\beta=\bar{s}-1$, where the lower bound is established for flat
Lipschitz domains and extends to bounded Lipschitz domains by
bi-Lipschitz change of variables. Methods that control the interior
residual through the embedding $L^2(\Omg)\hookrightarrow
H^{-1}(\Omg)$ realise the slightly weaker rate
\begin{equation}\label{eq:OR_rate}
  R^*_{\mtil,m}(\mathcal{U})_{H^1(\Omg)}
  \;\lesssim\; \mtil^{-\alpha} + m^{-\beta},
  \qquad
  \alpha = \tfrac{s}{2}
    - \bigl(\tfrac{1}{p}-\tfrac{1}{2}\bigr)_+,
  \quad
  \beta = \bar{s} - 1,
\end{equation}
which coincides with the unrestricted rate when $p\ge 2$. The
\emph{a priori} estimate in Theorem~\ref{thm:apriori} establishes
that the consistent loss $L^*_{\mathrm{sq},2}$ achieves the rate
in~\eqref{eq:OR_rate}, which is minimax optimal over
$\mathcal{F}$ and $\mathcal{G}$ for $p\ge 2$ and within a factor
$\mtil^{-(1/p-1/2)}$ of the lower bound otherwise. The experiments
of Section~\ref{sec:numerics} use $C^\infty$ data, for which the
two rates coincide.
 
We refer to the right-hand side of the energy
estimate~\eqref{eq:energy_equiv} as the
\emph{theoretical loss}
$\mathcal{L}_T(v)
:= \norm{f+\Delta v}_{H^{-1}(\Omg)}
  + \norm{g - v}_{H^{1/2}(\bOmg)}$.
The consistent loss $L^*_{\mathrm{sq},2}$ is a computable
surrogate for $\mathcal{L}_T$: by
Theorem~\ref{thm:apriori}, small values of
$L^*_{\mathrm{sq},2}(v)$ certify small values of
$\mathcal{L}_T(v)$, and hence small $H^1$ error.

\section{Error Analysis: norm and \emph{a priori} bounds}\label{sec:theory}

We introduce the piecewise polynomial machinery needed for
the proofs (\S\ref{sec:polyapprox}), establish that the
discrete $H^{1/2}(\bOmg)$ norm is equivalent to its
continuous counterpart on $\C^2$ level-set boundaries
(\S\ref{sec:norm_equiv}), and prove that
$L^*_{\mathrm{sq},2}$ controls the $H^1(\Omg)$ error at
the optimal recovery rate (\S\ref{sec:apriori}).

\subsection{Piecewise polynomial approximation}\label{sec:polyapprox}

Let $r > \max(\bar{s}, 1)$ denote a fixed polynomial degree. We assume
$r \mid m$ throughout and  the remainder case modifies constants only by factors depending on $r$. Let $\mathbb{T}_L := \mathbb{R}/L\mathbb{Z}$ denote the periodic interval of length $L$.

The $m$ samples $s_j = jL/m$, $j = 0, 1, \ldots, m-1$, lie on
$\mathbb{T}_L$. We partition the sample set into $m_0 := m/r$
consecutive blocks of $r+1$ points, with neighbouring blocks sharing
one endpoint, giving the coarse partition
$\tilde{I}_k = [\tilde{s}_{k-1}, \tilde{s}_k]$, $k = 1, \ldots, m_0$,
with block endpoints
\begin{equation}\label{eq:cell_endpoints}
  \tilde{s}_k = s_{kr} = k\tilde{h},
  \qquad
  \tilde{h} = rL/m,
\end{equation}
subject to the periodic identification $\tilde{s}_0 = \tilde{s}_{m_0}$.
Each cell $\tilde{I}_k$ contains $r+1$ sample points, sufficient to
determine a polynomial of degree $r$.

% Let $r > \max(\bar{s},1)$ be a fixed polynomial degree. We assume
% $r \mid m$ without loss of generality, since the remainder case alters
% constants only by factors depending on $r$.

% Let us introduce a space $\mathbb{T}_L := \mathbb{R}/L\mathbb{Z}$, the periodic interval of length $L$.

% The $m$ samples $s_j = jL/m$, $j = 0, 1, \ldots, m-1$, lie on the
% periodic interval $\mathbb{T}_L = [0,L]/(0 \sim L)$. Grouping them
% into $m_0 := m/r$ consecutive blocks of $r+1$ points, neighbouring
% blocks sharing one endpoint, we obtain the coarse partition
% $\tilde{I}_k = [\tilde{s}_{k-1}, \tilde{s}_k]$, $k = 1, \ldots, m_0$,
% with block endpoints
% \begin{equation}\label{eq:cell_endpoints}
%   \tilde{s}_k = s_{kr} = k\tilde{h},
%   \qquad
%   \tilde{h} = rL/m,
% \end{equation}
% periodic at $\tilde{s}_0 = \tilde{s}_{m_0}$. Each coarse cell
% $\tilde{I}_k$ contains $r+1$ sample points, enough to determine a
% polynomial of degree $r$.

Denote by $V^r(\mathbb{T}_L)$ the space of continuous piecewise
polynomials of degree at most $r$ on $\{\tilde{I}_k\}_{k=1}^{m_0}$,
periodic at the endpoints. For each coarse cell $\tilde{I}_k$, the
affine map $\phi_k\colon [0,1] \to \tilde{I}_k$,
$\phi_k(\xi) = \tilde{s}_{k-1} + \tilde{h}\xi$, pulls back any
$S \in V^r(\mathbb{T}_L)$ to a polynomial of degree $r$ on $[0,1]$.

Let $S_k\colon \C(\mathbb{T}_L) \to V^r(\mathbb{T}_L)$ denote the
quasi-interpolation operator of \cite[\S6.2]{bonito2025} in its
periodic 1D form.  Denote $(S_k G)\big|_{\tilde{I}_k}$ as the degree-$r$
Lagrange interpolant of $G$ at the $r+1$ sample points lying in
$\tilde{I}_k$. 

By construction, $S_k G$ interpolates $G$ at every
$s_j$, that is $G(s_j) = (S_k G)(s_j)$ for all $j$. The approximation
estimates
\begin{equation}\label{eq:quasi_interp}
  \norm{G - S_k G}_{\C(\mathbb{T}_L)}
    \le C\,\norm{G}_{\mathrm{Tr}(B)}\,m^{-\bar{s}},
  \qquad
  \norm{G - S_k G}_{H^{1/2}(\mathbb{T}_L)}
    \le C\,\norm{G}_{\mathrm{Tr}(B)}\,m^{-\beta},
\end{equation}
hold with $\beta = \bar{s} - 1$ and constants $C$ depending on $r$
and $L$ but not on $m$ or $G$. The $H^{1/2}$ rate is the periodic 1D
specialisation of \cite[Theorem~3.2 and~\S A.6.2]{bonito2025} with
$d = 2$ and $\bar{p} = 2$, while the $\C(\mathbb{T}_L)$ rate follows
from the embedding $\mathrm{Tr}(B^{\bar{s}}_\infty(L^2(\Omg)))
\hookrightarrow \C(\partial\Omg)$ combined with standard polynomial
approximation theory on quasi-uniform partitions
\cite{devore1993besov,BrennerScott2008}.

\subsection{Boundary norm equivalence}\label{sec:norm_equiv}

In $d=2$, $\bOmg$ is a one-dimensional closed curve. We adapt the
flat-domain argument of~\cite{bonito2025} (Lemma~6.2 and
Theorem~6.3) to this setting via an arc-length parametrisation,
using a \textit{Chord-arc} bound to relate Euclidean and arc-length
distances on the curve. The following lemma supplies this bound.

\begin{lemma}[\textbf{\textit{Chord-arc} equivalence}]\label{lem:chord_arc}
Let $\Gamma \subset \R^2$ be a simple closed $\C^2$ curve of length $L$,
parametrised by arc length $\gamma\colon [0,L] \to \Gamma$ with
$|\gamma'(s)|=1$. Let $\kappa_{\max} :=
\sup_{s\in[0,L]}|\kappa(s)|$ be the maximum curvature. Define the
periodic arc-length distance $d_\Gamma(s,t) :=
\min\bigl(|s-t|,\,L-|s-t|\bigr)$. Then there exists a constant
$c_\Gamma = c_\Gamma(\Gamma) \in (0,1]$ such that
\begin{equation}\label{eq:chord_arc}
  c_\Gamma\, d_\Gamma(s,t)
  \;\leq\; |\gamma(s) - \gamma(t)|
  \;\leq\; d_\Gamma(s,t)
\end{equation}
for all $s,t \in [0,L]$.
\end{lemma}

\begin{proof}
The upper bound is immediate: the chord never exceeds the arc. For the
lower bound, consider first the case $d_\Gamma(s,t) \leq
\delta := (2\kappa_{\max})^{-1}$. By periodicity of $\gamma$ we may
relabel so that $|t-s| = d_\Gamma(s,t)$, i.e.\ the shorter arc joining
$\gamma(s)$ and $\gamma(t)$ is realised by the direct (non-wrapping)
parameter interval. Taylor expansion with integral remainder gives
\[
  \gamma(t) = \gamma(s) + \gamma'(s)(t-s)
  + \int_s^t (t-\tau)\gamma''(\tau)\,d\tau.
\]
Since $|\gamma'(s)|=1$ and $|\gamma''(\tau)| = |\kappa(\tau)| \leq
\kappa_{\max}$, the reverse triangle inequality gives
$|\gamma(s)-\gamma(t)| \geq |t-s| -
\tfrac{1}{2}\kappa_{\max}|t-s|^2$. Since $|t-s|\leq\delta =
(2\kappa_{\max})^{-1}$, the second term is at most $|t-s|/4$, giving
\begin{equation}\label{eq:chord_local}
  |\gamma(s)-\gamma(t)| \geq \tfrac{3}{4}\,d_\Gamma(s,t).
\end{equation}
For the case $d_\Gamma(s,t) > \delta$, simplicity of $\Gamma$ ensures
that $\gamma(s) = \gamma(t)$ only when $d_\Gamma(s,t)=0$, so the
continuous function
$\Phi(s,t) := |\gamma(s)-\gamma(t)|/d_\Gamma(s,t)$ is well-defined
and strictly positive on the compact set
$\{(s,t)\in[0,L]^2 : d_\Gamma(s,t) \geq \delta\}$ and therefore attains
a positive minimum $c_1 > 0$. Taking $c_\Gamma := \min(3/4, c_1)$
completes the proof.
\end{proof}

\begin{lemma}[\textbf{Polynomial seminorm equivalence on $\mathbb{T}_L$}]
\label{lem:poly_equiv}
For every $S \in V^r(\mathbb{T}_L)$,
\begin{equation}\label{eq:piecewise-equiv}
  |S|_{\dot{H}^{1/2}(\mathbb{T}_L)}
  \asymp |S|^*_{\dot{H}^{1/2}(\mathbb{T}_L)},
\end{equation}
with constants depending only on $r$ and $L$.
\end{lemma}

\begin{proof} We decompose
\begin{equation*}
  |S|^2_{\dot{H}^{1/2}(\mathbb{T}_L)}
  = \sum_{k,l=1}^{m_0} I_{k,l}(S),
  \qquad
  |S|^{*2}_{\dot{H}^{1/2}(\mathbb{T}_L)}
  = \frac{1}{m^2} \sum_{k,l=1}^{m_0} J_{k,l}(S),
\end{equation*}
with
\begin{equation*}
  I_{k,l}(S) := \int_{\tilde{I}_k}\!\int_{\tilde{I}_l}
    \frac{|S(s)-S(t)|^2}{d_\Gamma(s,t)^2}\,ds\,dt,
  \qquad
  J_{k,l}(S) := \sum_{\substack{s_i \in \tilde{I}_k,\, s_j \in \tilde{I}_l \\ i \ne j}}
    \frac{|S(s_i)-S(s_j)|^2}{d_\Gamma(s_i,s_j)^2}.
\end{equation*}
It suffices to
prove $I_{k,l}(S) \asymp m^{-2} J_{k,l}(S)$ for each cell-pair. We distinguish three types.

For pairs with $\tilde{I}_k \cap \tilde{I}_l \ne \emptyset$, other
than the wraparound pair $(\tilde{I}_{m_0}, \tilde{I}_1)$, the affine pullback $\phi_k(\xi) = \tilde{s}_{k-1} + \tilde{h}\xi$ reduces the
configuration to a polynomial of degree $r$ on $[0,1]$, sampled at
$r+1$ points, with kernel $|\xi - \eta|^{-2}$. The equivalence
follows from Case~1 of~\cite[Lemma~6.2]{bonito2025}. For pairs with $|k - l| \ge 2$ in the periodic sense, we have
$d_\Gamma(s, t) \ge \tilde{h}$, the kernel is bounded, and Case~2
of~\cite[Lemma~6.2]{bonito2025} applies.

The pair $(\tilde{I}_{m_0}, \tilde{I}_1)$ shares the identified
vertex $\tilde{s}_{m_0} = \tilde{s}_0$ and is not covered
by~\cite{bonito2025}. For $s \in \tilde{I}_{m_0}$ and
$t \in \tilde{I}_1$, the periodic distance is $d_\Gamma(s,t) = L - s + t$,
which equals $|\psi(s) - \psi(t)|$ for the affine map
\begin{equation*}
  \psi(\sigma) = \begin{cases}
    \sigma - \tilde{s}_{m_0-1} & \sigma \in \tilde{I}_{m_0}, \\
    \tilde{h} + \sigma & \sigma \in \tilde{I}_1,
  \end{cases}
\end{equation*}
sending $\tilde{I}_{m_0} \cup \tilde{I}_1$ onto $[0, 2\tilde{h}]$. The push-forward
$\hat{S}(\xi) := S(\psi^{-1}(\tilde{h}\xi))$ is a continuous
piecewise polynomial of degree $r$ on $[0,1] \cup [1,2]$,
continuity at $\xi = 1$ following from $S(\tilde{s}_{m_0}) =
S(\tilde{s}_0)$, with $r+1$ samples per sub-interval and kernel
$|\xi - \eta|^{-2}$. The pulled-back configuration falls under
Case~1 of~\cite[Lemma~6.2]{bonito2025}.
Summing over cell-pairs, the result \eqref{eq:piecewise-equiv} follows.
\end{proof}

With the \textit{chord-arc} equivalence and seminorm equivalence on $\mathbb{T}_L$ in hand, we establish the norm equivalence on a smooth curve in the following theorem.

\begin{theorem}[\textbf{Discrete $H^{1/2}$ norm equivalence on a smooth curve}]
\label{thm:h12_curve}
Let $\Omega \subset \mathbb{R}^2$ be a bounded domain with
$\partial\Omega$ a simple closed $C^2$ curve of length $L$ and
maximum curvature $\kappa_{\max}$. Let
$Z = \{z_1, \ldots, z_m\} \subset \partial\Omega$ be $m$ points
equally spaced in arc length on $\partial\Omega$. Let
$G = \mathrm{Tr}\, U(B^{\bar{s}}_\infty(L^2(\Omega)))$ with
$\bar{s} > 1$ be the trace model class, and set $\beta = \bar{s}-1$.
Then for every $g \in G$ and $m \ge 1$,
\begin{align}
  \|g\|_{H^{1/2}(\partial\Omega)}
  &\lesssim \|g\|^*_{H^{1/2}(\partial\Omega)}
   + \|g\|_{\mathrm{Tr}(B)}\, m^{-\beta},
  \label{eq:thm2-upper} \\
  \|g\|^*_{H^{1/2}(\partial\Omega)}
  &\lesssim \|g\|_{H^{1/2}(\partial\Omega)}
   + \|g\|_{\mathrm{Tr}(B)}\, m^{-\beta},
  \label{eq:thm2-lower}
\end{align}
where the constants depend on $\Gamma$, $\bar{s}$, and $r$.
\end{theorem}

\begin{proof}
\textbf{Step 1: Arc-length parametrisation.}
Let $\gamma\colon[0,L] \to \partial\Omega$ denote the arc-length
parametrisation, so that $|\gamma'(s)| = 1$ for all $s \in [0,L]$
and $\gamma(0) = \gamma(L)$. The closedness of $\partial\Omega$
permits the identification of $[0,L]$ with $\mathbb{T}_L$. The
boundary collocation points are
\begin{equation*}
  z_j = \gamma(s_j),\qquad s_j = jL/m,\qquad j = 1,\ldots,m,
\end{equation*}
where $s_m = L$ corresponds to $s_0 = 0$ in $\mathbb{T}_L$. For
$g \in C(\partial\Omega)$, define the pullback
$G := g \circ \gamma \in C(\mathbb{T}_L)$. The arc-length
parametrisation gives $d\sigma = ds$, so that
\begin{equation*}
  \|g\|^2_{L^2(\partial\Omega)} = \int_0^L |G(s)|^2\,ds,
  \qquad
  |g|^2_{H^{1/2}(\partial\Omega)}
  = \int_0^L\!\int_0^L
    \frac{|G(s)-G(t)|^2}{|\gamma(s)-\gamma(t)|^2}\,ds\,dt.
\end{equation*}

\medskip\noindent
\textbf{Step 2: Reduction to the periodic interval.}
By Lemma~\ref{lem:chord_arc},
$|\gamma(s) - \gamma(t)|^2 \asymp d_\Gamma(s,t)^2$, we have
\begin{equation}\label{eq:transfer}
  |g|^2_{H^{1/2}(\partial\Omega)}
    \asymp |G|^2_{\dot{H}^{1/2}(\mathbb{T}_L)},
  \qquad
  |g|^{*2}_{H^{1/2}(\partial\Omega)}
    \asymp |G|^{*2}_{\dot{H}^{1/2}(\mathbb{T}_L)}.
\end{equation}
The continuous and discrete seminorms on $\partial\Omega$ are
therefore equivalent to their periodic-interval counterparts on
$\mathbb{T}_L$, with constants depending on $c_\Gamma$.

\medskip\noindent
\textbf{Step 3: Extension to general $G \in C(\mathbb{T}_L)$.}
The quasi-interpolant $S_k$ satisfies the interpolation property
$G(s_j) = (S_k G)(s_j)$ for $j = 1, \ldots, m$. Since the discrete
seminorm $|\cdot|^*_{\dot{H}^{1/2}(\mathbb{T}_L)}$ depends only on
the values at sample points,
\begin{equation}\label{eq:interp_identity}
  |G|^*_{\dot{H}^{1/2}(\mathbb{T}_L)}
  = |S_k G|^*_{\dot{H}^{1/2}(\mathbb{T}_L)}
  \qquad \forall\, G \in C(\mathbb{T}_L).
\end{equation}
Applying Lemma~\ref{lem:poly_equiv} to
$S_k G \in V^r(\mathbb{T}_L)$,
\begin{equation}\label{eq:poly_equiv_applied}
  |S_k G|^*_{\dot{H}^{1/2}(\mathbb{T}_L)}
  \asymp |S_k G|_{\dot{H}^{1/2}(\mathbb{T}_L)}.
\end{equation}
The continuous $H^{1/2}$ quasi-interpolation
estimate~\eqref{eq:quasi_interp} together with the triangle
inequality gives
\begin{equation}\label{eq:triangle_continuous}
  \bigl|\, |S_k G|_{\dot{H}^{1/2}(\mathbb{T}_L)}
        - |G|_{\dot{H}^{1/2}(\mathbb{T}_L)} \,\bigr|
  \le C\, \|G\|_{\mathrm{Tr}(B)}\, m^{-\beta}.
\end{equation}
Chaining~\eqref{eq:interp_identity}--\eqref{eq:triangle_continuous}
gives
\begin{equation}\label{eq:semi_norm_equiv}
  \bigl|\, |G|^*_{\dot{H}^{1/2}(\mathbb{T}_L)}
        - |G|_{\dot{H}^{1/2}(\mathbb{T}_L)} \,\bigr|
  \lesssim \|G\|_{\mathrm{Tr}(B)}\, m^{-\beta},
\end{equation}
with constants depending only on $r$ and $L$.

\medskip\noindent
\textbf{Step 4: Composing the equivalences.}
It remains to combine the seminorm bound~\eqref{eq:semi_norm_equiv}
with an analogous $L^2$ bound on $\mathbb{T}_L$, then transfer the
combined estimate back to $\partial\Omega$ via the \textit{chord-arc}
equivalence~\eqref{eq:transfer}.

\smallskip\noindent
\underline{$L^2$ bound on $\mathbb{T}_L$:} The discrete $L^2$ norm
$\|G\|^{*2}_{L^2(\mathbb{T}_L)} := \frac{L}{m}\sum_{j=1}^{m}|G(s_j)|^2$
is the trapezoidal-rule approximation of
$\|G\|^2_{L^2(\mathbb{T}_L)} = \int_0^L |G(s)|^2\,ds$ on the uniform
grid $\{s_j\}_{j=1}^m$. Standard 1D quadrature theory on
quasi-uniform partitions
\cite[Theorem~4.4.4]{BrennerScott2008} gives, for any
$G \in \mathrm{Tr}(B)$ with $\bar{s} > 1/2$,
\begin{equation}\label{eq:L2_disc}
  \bigl|\, \|G\|^*_{L^2(\mathbb{T}_L)} - \|G\|_{L^2(\mathbb{T}_L)} \,\bigr|
  \le C\,\|G\|_{\mathrm{Tr}(B)}\,m^{-(\bar{s} - 1/2)}
  \le C\,\|G\|_{\mathrm{Tr}(B)}\,m^{-\beta},
\end{equation}
the second inequality using $\bar{s} - 1/2 \ge \bar{s} - 1 = \beta$
for $\bar{s} \ge 1/2$, which holds under our hypothesis $\bar{s} > 1$.

\smallskip\noindent
\underline{Combining seminorm and $L^2$:}
Set $R := C\,\|G\|_{\mathrm{Tr}(B)}\,m^{-\beta}$. Estimates
\eqref{eq:semi_norm_equiv} and~\eqref{eq:L2_disc} give
\begin{equation*}
  \bigl|\, |G|^*_{\dot{H}^{1/2}(\mathbb{T}_L)}
        - |G|_{\dot{H}^{1/2}(\mathbb{T}_L)} \,\bigr| \le R,
  \qquad
  \bigl|\, \|G\|^*_{L^2(\mathbb{T}_L)}
        - \|G\|_{L^2(\mathbb{T}_L)} \,\bigr| \le R.
\end{equation*}
Applying the elementary inequality $|a - b| \le |a^2 - b^2|/(a + b)$
for $a, b \ge 0$ with $a + b > 0$, together with $(p + q)^2 \le 2(p^2 + q^2)$,
gives
\begin{equation*}
  \bigl|\, \|G\|^*_{H^{1/2}(\mathbb{T}_L)}
        - \|G\|_{H^{1/2}(\mathbb{T}_L)} \,\bigr|
  \le \sqrt{2}\,R
  = \sqrt{2}\,C\,\|G\|_{\mathrm{Tr}(B)}\,m^{-\beta}.
\end{equation*}

\smallskip\noindent
% \underline{Transfer back to $\partial\Omega$:}
% By the chord-arc equivalence~\eqref{eq:transfer},
% $\|G\|_{H^{1/2}(\mathbb{T}_L)} \asymp \|g\|_{H^{1/2}(\partial\Omega)}$
% and $\|G\|^*_{H^{1/2}(\mathbb{T}_L)} \asymp \|g\|^*_{H^{1/2}(\partial\Omega)}$,
% both with constants absorbed into $c_\Gamma$. Substituting gives
% \eqref{eq:thm2-upper}--\eqref{eq:thm2-lower}, the constants depending
% on $\Gamma$, $\bar{s}$, and $r$.
\underline{Transfer back to $\partial\Omega$:}
The continuous equivalence
$\|G\|_{H^{1/2}(\mathbb{T}_L)} \asymp \|g\|_{H^{1/2}(\partial\Omega)}$
follows from the chord-arc equivalence~\eqref{eq:transfer} with
constants depending only on $c_\Gamma$. For the discrete norms, the
trapezoidal convention $\|G\|^{*2}_{L^2(\mathbb{T}_L)}
= (L/m)\sum_{j=1}^m |G(s_j)|^2$ in~\eqref{eq:L2_disc} differs from
Definition~\ref{def:h12} by the Riemann factor $L$, so
$\|G\|^{*2}_{H^{1/2}(\mathbb{T}_L)}
= L\,\|g\|^{*2}_{L^2(\partial\Omega)}
+ |g|^{*2}_{H^{1/2}(\partial\Omega)}$ up to the chord-arc factor on
the seminorm. This gives $\|G\|^*_{H^{1/2}(\mathbb{T}_L)} \asymp
\|g\|^*_{H^{1/2}(\partial\Omega)}$ with constants depending on
$c_\Gamma$ and $L$. Substituting
gives~\eqref{eq:thm2-upper}--\eqref{eq:thm2-lower}, the constants
depending on $\Gamma$, $\bar{s}$, and $r$.
\end{proof}

\begin{remark}\label{rem:quasiuniform}
Theorem~\ref{thm:h12_curve} assumes equally-spaced boundary points
in arc length. The experiments use parametric sampling (uniform in
angle for the disk, uniform in $\theta$ for the flower), which
produces points that are equally spaced in arc length for the disk
and approximately so for the flower. A generalisation to
quasi-uniform point distributions, where the ratio of maximal to
minimal spacing is bounded, follows from the same argument with
modified constants.
\end{remark}
 
As a robustness check, we verify that the seminorm V-statistic
$|g|^{*2}_{H^{1/2}(\partial\Omega)}$ concentrates around its mean
$L^{-2}|g|^2_{H^{1/2}(\partial\Omega)}$ at the Monte Carlo rate even
under {\em i.i.d} uniform boundary sampling, where the quasi-uniform
argument of Theorem~\ref{thm:h12_curve} does not apply. The
following proposition is not invoked in Theorem~\ref{thm:apriori};
it is included to support the experimental observation that the
consistent loss remains effective under random sampling regimes
(Experiment~\ref{sec:exp6}).
 
\begin{proposition}[\textbf{Probabilistic $H^{1/2}$ seminorm
concentration under {\em i.i.d.}\ uniform boundary
sampling}]\label{prop:h12_rejection}
Adopt the geometric hypotheses of Theorem~\ref{thm:h12_curve}, and let
$\mathcal{G} = \mathrm{Tr}\,U(B^{\bar{s}}_\infty(L^2(\Omega)))$ with
$\bar{s} > 3/2$. Let $\{z_j\}_{j=1}^{m} \subset \partial\Omega$ be
drawn {\em i.i.d} uniformly with respect to arc-length on
$\partial\Omega$. Then for every $g \in \mathcal{G}$ and every
$\delta \in (0,1)$, with probability at least $1-\delta$,
\begin{equation}\label{eq:prop5}
  \bigl|\,L^2\,|g|^{*2}_{H^{1/2}(\partial\Omega)}
        - |g|^2_{H^{1/2}(\partial\Omega)}\,\bigr|
  \;\le\; C_\Gamma\,\|g\|^2_{\mathrm{Tr}(B)}\,(m\delta)^{-1/2},
\end{equation}
where $C_\Gamma$ depends only on $L$, $\kappa_{\max}$, $\bar{s}$, and
the \textit{chord-arc} constant $c_\Gamma$ of
Lemma~\ref{lem:chord_arc}.
\end{proposition}
 
\begin{proof} 
Set $K(z,z') := |g(z)-g(z')|^2/|z-z'|^2$ (with $K(z,z):=0$) and the
V-statistic
$T := |g|^{*2}_{H^{1/2}(\partial\Omega)}
= m^{-2}\sum_{i\ne j}K(z_i,z_j)$.
 
\smallskip\noindent\textbf{Kernel bound.}
By Lemma~\ref{lem:chord_arc}, $|z-z'| \ge c_\Gamma\,d_\Gamma(s,t)$,
so the mean-value theorem along the arc-length parametrisation gives
$|g(z)-g(z')| \le \|g\|_{C^1(\partial\Omega)}\,d_\Gamma
\le c_\Gamma^{-1}\|g\|_{C^1}|z-z'|$. Therefore, we have
\begin{equation}\label{eq:K_bound_prop5}
  K(z,z') \;\le\; M_g := c_\Gamma^{-2}\,\|g\|^2_{C^1(\partial\Omega)}.
\end{equation}
 
\smallskip\noindent\textbf{Mean and variance.}
The $z_i$ being {\em i.i.d.}\ with arc-length density $1/L$,
$\mathbb{E}[K(z_1,z_2)] = L^{-2}|g|^2_{H^{1/2}(\partial\Omega)}$ and
$\mathbb{E}[T] = \tfrac{m-1}{m}\,L^{-2}|g|^2_{H^{1/2}(\partial\Omega)}$.
The Hoeffding decomposition of $T$ (an order-2 U-statistic with
diagonal excluded, normalised by $m^{-2}$) gives
\[
  \mathrm{Var}(T)
  = \frac{2(m-1)\,\mathrm{Var}(K) + 4(m-1)(m-2)\,\zeta_1}{m^3},
  \qquad
  \zeta_1 := \mathrm{Var}\!\bigl(\mathbb{E}[K(z,z')\mid z]\bigr).
\]
Jensen's inequality gives $\zeta_1 \le \mathrm{Var}(K)$, hence
$\mathrm{Var}(T) \le 4\,\mathrm{Var}(K)/m$ for $m \ge 2$
(cf.~\cite{serfling1980}, \S5.2.1). Combined with~\eqref{eq:K_bound_prop5}
and $\mathrm{Var}(K) \le \mathbb{E}[K^2] \le M_g\,\mathbb{E}[K]$:
\begin{equation}\label{eq:var_T_prop5}
  \mathrm{Var}(T)
  \;\le\;
  \frac{4\,M_g\,|g|^2_{H^{1/2}(\partial\Omega)}}{L^2\,m}.
\end{equation}
 
\smallskip\noindent\textbf{Chebyshev's inequality.}
With probability at least $1-\delta$,
\[
  |T - \mathbb{E}[T]|
  \;\le\; \sqrt{\mathrm{Var}(T)/\delta}
  \;\le\; \frac{2\,\sqrt{M_g}\,|g|_{H^{1/2}(\partial\Omega)}}
              {L\,\sqrt{m\delta}}.
\]
Adding the deterministic bias
$|\mathbb{E}[T] - L^{-2}|g|^2_{H^{1/2}}| = L^{-2}|g|^2_{H^{1/2}}/m$
and multiplying through by $L^2$:
\begin{equation}\label{eq:T_sq_almost_prop5}
  \bigl|\,L^2\,T - |g|^2_{H^{1/2}(\partial\Omega)}\,\bigr|
  \;\le\;
  \frac{2\,L\,\sqrt{M_g}\,|g|_{H^{1/2}(\partial\Omega)}}{\sqrt{m\delta}}
  \;+\;
  \frac{|g|^2_{H^{1/2}(\partial\Omega)}}{m}.
\end{equation}
 
\smallskip\noindent\textbf{Closing with the Besov embedding.}
From the kernel bound~\eqref{eq:K_bound_prop5} integrated over
$\partial\Omega \times \partial\Omega$, one can obatin that
$$|g|_{H^{1/2}(\partial\Omega)} \le L\,c_\Gamma^{-1}\,\|g\|_{C^1(\partial\Omega)}.$$
By using above inequality, the bound \eqref{eq:T_sq_almost_prop5} is deduced by
\begin{equation}\label{eq:T_sq_almost_prop05}
  \bigl|\,L^2\,T - |g|^2_{H^{1/2}(\partial\Omega)}\,\bigr|
  \;\le\; \dfrac{ \|g\|^2_{C^1(\partial\Omega)}}{\sqrt{m\delta}}.
\end{equation}
% Therefore, both the  terms on the right of~\eqref{eq:T_sq_almost_prop5} are
% bounded by a constant multiple of
% $\|g\|^2_{C^1(\partial\Omega)}/\sqrt{m\delta}$ for $\delta \le 1$
% (the $m^{-1}$ bias is absorbed into the leading $(m\delta)^{-1/2}$).
Finally, 
by using Besov embedding 
\[
\mathrm{Tr}(B^{\bar{s}}_\infty) \hookrightarrow C^1(\partial\Omega), \quad \bar{s} > 3/2,
\]
we have 
% $\mathrm{Tr}(B^{\bar{s}}_\infty) \hookrightarrow C^1(\partial\Omega)$
% for $\bar{s} > 3/2$ gives
$\|g\|_{C^1(\partial\Omega)} \le C(\Omega,\bar{s})\,\|g\|_{\mathrm{Tr}(B)}$.
% yielding~\eqref{eq:prop5} with
% $C_\Gamma := C_\Gamma(L,c_\Gamma,\kappa_{\max},\bar{s})$.
% This completes the proof.
Using this bound, we obtain \eqref{eq:prop5} with   $C_\Gamma := C_\Gamma(L,c_\Gamma,\kappa_{\max},\bar{s})$.
This completes the proof.
\end{proof}

%\newpage
\subsection{The \emph{a priori} error bound}\label{sec:apriori}
With the discrete norm equivalence established, we can now state the
main \emph{a priori} result: the consistent loss $L^*_{\mathrm{sq},2}$
controls the $H^1$ error on the cut domain, up to the optimal
recovery rate. Before stating the theorem, we provide auxiliary
discretisation results on the cut domain interior that justify
hypothesis~\eqref{eq:thm_L2_disc} below.

\begin{figure}[ht]
  \centering
  \includegraphics[width=0.7\textwidth]{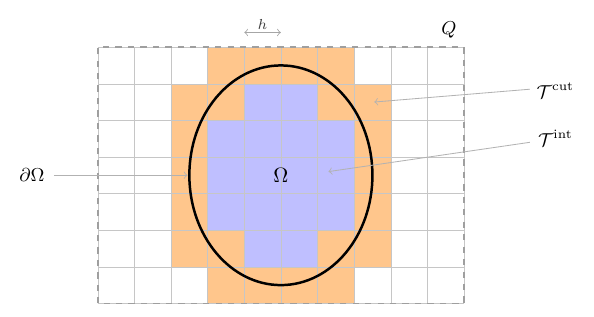}
  \caption{Cut-cell decomposition of $\Omega \subset Q$: interior cells
  $\mathcal{T}^{\mathrm{int}}$ (teal) and cut cells $\mathcal{T}^{\mathrm{cut}}$
  (red, crossing $\partial\Omega$).}
  \label{fig:cut_cell}
\end{figure}

\begin{proposition}[\textbf{$L^2$ discretisation on cut domains}]
\label{prop:l2_cut}
Let $\Omega\subset\mathbb{R}^2$ be a bounded domain with
$\partial\Omega$ a $\C^2$ level-set curve of length~$L$; fix
$1\le p\le\infty$, $s>2/p$, $r>\max(s,1)$, and write
$B=B^s_p(\Omega)$. Let $Q\supset\Omega$ be an axis-aligned bounding
box and $G_{k,r}\subset Q$ the tensor-product grid of mesh size
$h=(\mathrm{diam}\,Q)\cdot 2^{-k}/(r-1)$. Set
$\tilde{X}:=G_{k,r}\cap\Omega$, $\tilde{m}:=|\tilde{X}|$, and
$\|w\|^{*2}_{L^2(\Omega)}:=(|\Omega|/\tilde{m})\sum_{x_i\in\tilde{X}}|w(x_i)|^2$.
Then for every $w\in B$,
\begin{equation}\label{eq:prop3}
  \bigl|\,\|w\|_{L^2(\Omega)}-\|w\|^*_{L^2(\Omega)}\bigr|
  \;\le\; C_{\mathrm{int}}\,\|w\|_B\,\tilde{m}^{-\alpha_\Omega},
  \quad
  \alpha_\Omega
  := \min\!\bigl(\tfrac{s}{2}
    -(\tfrac{1}{p}-\tfrac{1}{2})_+,\;\tfrac{1}{4}\bigr),
\end{equation}
with $C_{\mathrm{int}}$ depending on $\Omega$, $s$, $p$, $r$, and the
Stein extension constant $C_E$.
\end{proposition}
 
\begin{proof}
The argument extends~\cite[Lemma~6.1]{bonito2025} from the box $Q$ to
the cut domain by isolating the cut-cell contribution.
 
\medskip\noindent\textbf{Setup.}
Let $\hat{w}:=Ew$ be the Stein extension~\cite{Stein1970} with
$\|\hat{w}\|_{B^s_p(\mathbb{R}^2)}\le C_E\|w\|_B$, and
$\|\hat{w}\|_{\C(\bar Q)}\lesssim\|w\|_B$ via
$B^s_p\hookrightarrow\C$ for $s>2/p$. Let
$P:=S^*_k(\hat{w})\in V^r(\mathcal{T}_k)$ be the BDPS interpolant on
the $N=(r\cdot 2^k)^2$ nodes of $G_{k,r}$. By~\cite[Theorem~2.1]{bonito2025},
\begin{equation}\label{eq:approx_P}
  \|\hat{w}-P\|_{L^2(Q)}\;\lesssim\;\|w\|_B\,N^{-\alpha},
  \qquad
  \alpha:=\tfrac{s}{2}-(\tfrac{1}{p}-\tfrac{1}{2})_+,
\end{equation}
and $\|P\|_{\C(\bar Q)}\lesssim\|w\|_B$ by stability of $S^*_k$
(Lebesgue constant $\Lambda_r$). Partition the cells meeting
$\bar\Omega$ into interior cells $\mathcal{T}^{\mathrm{int}}$
(contained in $\Omega$) and cut cells $\mathcal{T}^{\mathrm{cut}}$
(meeting $\bOmg$); set
$\Omega^{\mathrm{int}}:=\bigcup\mathcal{T}^{\mathrm{int}}$ and
$\tilde{X}_{\mathrm{int}}:=G_{k,r}\cap\Omega^{\mathrm{int}}$
(see Figure~\ref{fig:cut_cell}). Since $\bOmg\in\C^2$ has length~$L$,
\begin{equation}\label{eq:cut_count}
  |\mathcal{T}^{\mathrm{cut}}|=O(\tilde m^{1/2}),
  \quad
  |\Omega\setminus\Omega^{\mathrm{int}}|=O(\tilde m^{-1/2}),
  \quad
  \tfrac{|\Omega^{\mathrm{int}}|}{|\tilde X_{\mathrm{int}}|}
  =\tfrac{h^2}{r^2}
  =\tfrac{|\Omega|}{\tilde m}\bigl(1+O(\tilde m^{-1/2})\bigr).
\end{equation}
 
\medskip\noindent\textbf{Interior and cut bounds.}
On $\Omega^{\mathrm{int}}$ (a union of complete cells), the BDPS
norm equivalence~\cite[Lemma~6.1]{bonito2025} gives
$\bigl|\|P\|^2_{L^2(\Omega^{\mathrm{int}})}-(h^2/r^2)\sum_{\tilde X_{\mathrm{int}}}P(x_i)^2\bigr|
\lesssim\|w\|^2_B\,\tilde m^{-2\alpha}$. Substituting~\eqref{eq:cut_count}
and using $P(x_i)=w(x_i)$ at
$\tilde X_{\mathrm{int}}\subset\tilde X$:
\begin{equation}\label{eq:int_to_disc}
  \|P\|^2_{L^2(\Omega^{\mathrm{int}})}
  \;\le\;\|w\|^{*2}_{L^2(\Omega)}\bigl(1+O(\tilde m^{-1/2})\bigr)
   + C\,\|w\|^2_B\,\tilde m^{-2\alpha}.
\end{equation}
On the cut region, H\"older gives
$\|P\|^2_{L^2(\Omega\setminus\Omega^{\mathrm{int}})}
\le\|P\|^2_{\C(\bar Q)}\,|\Omega\setminus\Omega^{\mathrm{int}}|
\lesssim\|w\|^2_B\,\tilde m^{-1/2}$. Adding,
\begin{equation}\label{eq:full_P_bound}
  \|P\|^2_{L^2(\Omega)}
  \;\le\;\|w\|^{*2}_{L^2(\Omega)}\bigl(1+O(\tilde m^{-1/2})\bigr)
   + C\,\|w\|^2_B\,\tilde m^{-\min(2\alpha,\,1/2)}.
\end{equation}
 
\medskip\noindent\textbf{Assembly.}
The uniform bound $\|w\|^*_{L^2(\Omega)}\le|\Omega|^{1/2}\|w\|_{\C(\bar\Omega)}
\lesssim\|w\|_B$ absorbs the ratio correction. Taking square roots
in~\eqref{eq:full_P_bound} and combining with the triangle inequality
$\|w-P\|_{L^2(\Omega)}\le\|\hat w-P\|_{L^2(Q)}\lesssim\|w\|_B\,\tilde m^{-\alpha}$:
\[
  \|w\|_{L^2(\Omega)}\le\|w\|^*_{L^2(\Omega)}
   + C\,\|w\|_B\,\tilde m^{-\min(\alpha,\,1/4)},
\]
the forward direction of~\eqref{eq:prop3} with
$\alpha_\Omega=\min(\alpha,\,1/4)$.
 
For the reverse, splitting
$\tilde X=\tilde X_{\mathrm{int}}\cup(\tilde X\setminus\tilde X_{\mathrm{int}})$:
the interior sum is bounded by the reverse of~\eqref{eq:int_to_disc},
and the cut sum has at most
$r^2|\mathcal{T}^{\mathrm{cut}}|=O(\tilde m^{1/2})$ nodes weighted by
$|\Omega|/\tilde m=O(\tilde m^{-1})$, contributing
$\lesssim\|w\|^2_B\,\tilde m^{-1/2}$ via
$\|P\|_{\C(\bar Q)}\lesssim\|w\|_B$. The same square-root and triangle
argument then gives
$\|w\|^*_{L^2(\Omega)}-\|w\|_{L^2(\Omega)}\le C\,\|w\|_B\,\tilde m^{-\min(\alpha,\,1/4)}$.
\end{proof}

As an analogous robustness check for the interior $L^2$ discretisation,
we consider {\em i.i.d} uniform sampling on $\Omega$. The following
proposition parallels Proposition~\ref{prop:h12_rejection}; it is not
invoked in Theorem~\ref{thm:apriori}, but supports the experimental
observation that the consistent loss remains effective under random
sampling regimes (Experiment~\ref{sec:exp6}).
 
\begin{proposition}[\textbf{Probabilistic $L^2$ discretisation under
{\em i.i.d} uniform sampling}]\label{prop:l2_rejection}
Let $\Omega\subset\mathbb{R}^2$ be a bounded domain and let
$\{x_i\}_{i=1}^{\tilde m}\subset\Omega$ be drawn {\em i.i.d} uniformly.
Set $\|w\|^{*2}_{L^2(\Omega)} := \tfrac{|\Omega|}{\tilde m}
\sum_{i=1}^{\tilde m}|w(x_i)|^2$.
For every $w\in L^\infty(\Omega)$ with $\|w\|_{L^2(\Omega)}>0$ and
every $\delta\in(0,1)$, with probability at least $1-\delta$,
\begin{equation}\label{eq:prop4}
  \bigl|\,\|w\|_{L^2(\Omega)} - \|w\|^*_{L^2(\Omega)}\,\bigr|
  \;\le\; |\Omega|^{1/2}\,\|w\|_{L^\infty(\Omega)}\,(\tilde m\,\delta)^{-1/2}.
\end{equation}
For $w \in U(B^s_p(\Omega))$ with $s > 2/p$, the embedding
$B^s_p \hookrightarrow \C(\bar\Omega)$ replaces $\|w\|_{L^\infty}$
in~\eqref{eq:prop4} by $C(\Omega,s,p)\,\|w\|_B$, yielding the Monte
Carlo rate $\tilde m^{-1/2}$ uniformly over the model class.
\end{proposition}
 
\begin{proof}
Let $X_i := |w(x_i)|^2$ and $S_{\tilde m} := \tilde m^{-1}\sum_{i=1}^{\tilde m} X_i$.
Since the $x_i$ are {\em i.i.d} uniform on $\Omega$,
\[
  \mathbb{E}[X_1]
  = \frac{1}{|\Omega|}\int_\Omega|w|^2
  = \frac{\|w\|^2_{L^2(\Omega)}}{|\Omega|},
  \qquad
  \mathrm{Var}(X_1)
  \le \mathbb{E}[X_1^2]
  \le \frac{\|w\|^2_{L^\infty(\Omega)}\,\|w\|^2_{L^2(\Omega)}}{|\Omega|},
\]
the second bound using
$\int_\Omega|w|^4\le\|w\|^2_{L^\infty}\int_\Omega|w|^2$. Hence
$\mathrm{Var}(S_{\tilde m})=\mathrm{Var}(X_1)/\tilde m\le
\|w\|^2_{L^\infty}\,\|w\|^2_{L^2(\Omega)}/(|\Omega|\,\tilde m)$, and
Chebyshev's inequality yields, with probability at least $1-\delta$,
$|S_{\tilde m}-\mathbb{E}[X_1]|\le
\sqrt{\mathrm{Var}(S_{\tilde m})/\delta}\le
\|w\|_{L^\infty}\,\|w\|_{L^2(\Omega)}/\sqrt{|\Omega|\,\tilde m\,\delta}$.
Multiplying through by $|\Omega|$,
\begin{equation}\label{eq:prop4_sq}
  \bigl|\,\|w\|^{*2}_{L^2(\Omega)} - \|w\|^2_{L^2(\Omega)}\,\bigr|
  \;\le\; |\Omega|^{1/2}\,\|w\|_{L^\infty(\Omega)}\,\|w\|_{L^2(\Omega)}\,
          (\tilde m\,\delta)^{-1/2}.
\end{equation}
The identity $|a-b|(a+b) = |a^2-b^2|$ with
$a = \|w\|^*_{L^2(\Omega)} \ge 0$, $b = \|w\|_{L^2(\Omega)} > 0$ gives
$|a-b| \le |a^2-b^2|/b$. Applied to~\eqref{eq:prop4_sq}, the factor
$\|w\|_{L^2(\Omega)}$ on the right cancels, establishing~\eqref{eq:prop4}.
The Besov special case follows from $\|w\|_{L^\infty(\Omega)} \le
C(\Omega,s,p)\,\|w\|_B$ via the embedding for $s>2/p$.
\end{proof}

\begin{theorem}[\textbf{A\emph{ priori} error bound on cut
domains}]\label{thm:apriori}
Adopt the model classes~\eqref{eq:OR_classes} and the consistent
loss~\eqref{eq:Lc2}, with $r > \max(\bar{s}, 1)$. Let
$\tilde{X} = \{x_i\}_{i=1}^{\mtil}\subset\Omg$ be interior sites
satisfying the discrete $L^2$ comparison
\begin{equation}\label{eq:thm_L2_disc}
  \bigl|\,\norm{w}_{L^2(\Omg)} - \norm{w}^*_{L^2(\Omg)}\,\bigr|
  \;\le\; C_{\mathrm{int}}\,\norm{w}_B\,\mtil^{-\alpha_\Omega},
  \quad
  \alpha_\Omega := \min\!\bigl(\tfrac{s}{2}
    -(\tfrac{1}{p}-\tfrac{1}{2})_+,\,\tfrac{1}{4}\bigr),
\end{equation}
for every $w \in B^s_p(\Omega)$, where
$\norm{w}^{*2}_{L^2(\Omg)} := (|\Omg|/\mtil)\sum_{i=1}^{\mtil}|w(x_i)|^2$;
and let $Z = \{z_j\}_{j=1}^m \subset \bOmg$ be $m$ points equally
spaced in arc length. Then for every $v \in H^1(\Omg)$ with
$\Delta v \in \C(\bar\Omg)$,
\begin{equation}\label{eq:apriori_bound}
  \norm{u - v}_{H^1(\Omg)}
  \;\le\; C\,\Bigl[\, L^*_{\mathrm{sq},2}(v)^{1/2}
    + \bigl(1 + \norm{v}_{\mathcal{U}}\bigr)\,
      \bigl(\mtil^{-\alpha_\Omega} + m^{-\beta}\bigr)\,\Bigr],
\end{equation}
where $\beta := \bar{s} - 1$,
$\norm{v}_{\mathcal{U}} := \max\bigl\{\norm{\Delta v}_B,\,
\norm{\mathrm{Tr}(v)}_{\mathrm{Tr}(B)}\bigr\}$, and the constant $C$
depends on $\Omg$, $L$, $\kappa_{\max}$, $s$, $\bar{s}$, and $r$ but
not on $\mtil$, $m$, $v$, or $u$.
Hypothesis~\eqref{eq:thm_L2_disc} is verified deterministically by
Proposition~\ref{prop:l2_cut} (with $\tilde X = G_{k,r}\cap\Omg$), or
in probability by Proposition~\ref{prop:l2_rejection} with the
sharper rate $\alpha_\Omega = 1/2$.
\end{theorem}
 
\begin{proof}
Throughout, $\lesssim$ absorbs constants depending on the Poincar\'e
constant of $\Omg$, the area $|\Omg|$, the arc length $L$, the
curvature $\kappa_{\max}$, the chord-arc constant $c_\Gamma$ of
Lemma~\ref{lem:chord_arc}, $s$, $\bar{s}$, $r$, and the
$C_{\mathrm{int}}$ of~\eqref{eq:thm_L2_disc}.
 
\medskip\noindent\textbf{Step 1: Energy estimate.}
Since $\bOmg \in \C^2$, $\Omg$ is Lipschitz; the difference $u-v$
solves $-\Delta(u-v) = f + \Delta v$ in $\Omg$ with
$(u-v)|_{\bOmg} = g - \mathrm{Tr}(v)$, so the standard
\emph{a priori} bound for the Poisson
equation~\cite[Eq.~(1.4)]{bonito2025} gives
\begin{equation}\label{eq:energy_step}
  \norm{u-v}_{H^1(\Omg)}
  \;\lesssim\; \norm{f+\Delta v}_{H^{-1}(\Omg)}
   + \norm{g-\mathrm{Tr}(v)}_{H^{1/2}(\bOmg)}.
\end{equation}
 
\medskip\noindent\textbf{Step 2: Interior term.}
Poincar\'e on $H^1_0(\Omg)$ gives the embedding
$L^2(\Omg)\hookrightarrow H^{-1}(\Omg)$. Applying~\eqref{eq:thm_L2_disc}
with $w = f + \Delta v$ and using $\norm{f}_B \le 1$,
\begin{equation}\label{eq:interior_chain}
  \norm{f+\Delta v}_{H^{-1}(\Omg)}
  \;\lesssim\; \norm{f+\Delta v}^*_{L^2(\Omg)}
   + (1+\norm{v}_{\mathcal{U}})\,\mtil^{-\alpha_\Omega}.
\end{equation}
 
\medskip\noindent\textbf{Step 3: Boundary term.}
Theorem~\ref{thm:h12_curve} (discrete $H^{1/2}$ equivalence on the
$\C^2$ curve via the \textit{Chord-arc} reduction of
Lemma~\ref{lem:chord_arc}) and $\norm{g}_{\mathrm{Tr}(B)} \le 1$ give
\begin{equation}\label{eq:boundary_chain}
  \norm{g-\mathrm{Tr}(v)}_{H^{1/2}(\bOmg)}
  \;\lesssim\; \norm{g-v}^*_{H^{1/2}(\bOmg)}
   + (1+\norm{v}_{\mathcal{U}})\,m^{-\beta},
\end{equation}
the implicit constant absorbing $c_\Gamma$ and $L$.
 
\medskip\noindent\textbf{Step 4: Assembly.}
Substituting~\eqref{eq:interior_chain}--\eqref{eq:boundary_chain}
into~\eqref{eq:energy_step} and applying Cauchy--Schwarz $a + b \le
\sqrt{2}(a^2+b^2)^{1/2}$ to the two discrete terms,
\[
  \norm{f+\Delta v}^*_{L^2(\Omg)}
  + \norm{g-v}^*_{H^{1/2}(\bOmg)}
  \;\le\; \sqrt{2}\,\bigl(\norm{f+\Delta v}^{*2}_{L^2(\Omg)}
    + \norm{g-v}^{*2}_{H^{1/2}(\bOmg)}\bigr)^{1/2}.
\]
The $L^2$-side equals
$|\Omg|\cdot\mtil^{-1}\sum_i[f(x_i)+\Delta v(x_i)]^2$, so
comparing with~\eqref{eq:Lc2} gives
\[
  \norm{f+\Delta v}^{*2}_{L^2(\Omg)}
  + \norm{g-v}^{*2}_{H^{1/2}(\bOmg)}
  \;\le\; \max(|\Omg|,1)\,L^*_{\mathrm{sq},2}(v),
\]
and absorbing $\sqrt{2\max(|\Omg|,1)}$ into the implicit constant
yields~\eqref{eq:apriori_bound}.
\end{proof}

\begin{remark}[Role of Theorem~\ref{thm:h12_curve}]
Within the proof of Theorem~\ref{thm:apriori}, the geometry of $\bOmg$
enters only in Step~3 via Theorem~\ref{thm:h12_curve}; Steps~1 and~2
hold on any bounded Lipschitz domain once
hypothesis~\eqref{eq:thm_L2_disc} is supplied, the cut-cell geometry
being absorbed in advance by Proposition~\ref{prop:l2_cut}. The novelty
relative to~\cite{bonito2025} is therefore concentrated in the boundary
discretisation, which the \textit{Chord-arc} reduction of
Theorem~\ref{thm:h12_curve} provides.
\end{remark}
 
\begin{remark}[Optimality of the recovery rate]\label{rem:optimality}
The boundary rate $m^{-\beta}$ in Theorem~\ref{thm:apriori} matches
the unrestricted optimal recovery rate
of~\cite[Theorem~3.1]{bonito2025}: the boundary is sampled along the
curve directly, with no cut-cell analogue. The interior rate
$\mtil^{-\alpha_\Omega}$ with $\alpha_\Omega = \min(\alpha,\,1/4)$
reflects a geometric obstruction specific to \emph{tensor-grid}
sampling on $\Omg$: the $O(\mtil^{1/2})$ cells intersecting $\bOmg$
contribute an unavoidable $\mtil^{-1/4}$ error
(Proposition~\ref{prop:l2_cut}). Two routes break this barrier:
\textup{(i)}~sub-cell quadrature on cut cells, as in unfitted finite
element methods~\cite{li2026}, at the cost of explicit mesh
infrastructure; \textup{(ii)}~i.i.d.\ uniform sampling on $\Omg$
(Proposition~\ref{prop:l2_rejection}), which recovers the Monte Carlo
rate $\mtil^{-1/2}$ in probability at no additional mesh cost. The
standard loss $L_{\mathrm{sq},1}$, by contrast, achieves no rate in
$H^1$: its $L^2(\bOmg)$ boundary term does not control
$\norm{g-v}_{H^{1/2}(\bOmg)}$, and the error saturates at the
boundary oscillation undetected by $L^2$
(cf.\ Section~\ref{sec:exp2}).
\end{remark}

\section{Numerical Discussion}
\label{sec:numerics}

Theorem~\ref{thm:h12_curve} guarantees a norm equivalence between
the discrete and continuous $H^{1/2}$ norms on a $\C^2$ boundary,
but it does not tell us how large the practical gain of the
consistent loss over the standard one will be at a finite sample
size. The experiments here address that question. The four loss
functions share the same network, optimiser and training budget
throughout, so any difference in the $H^1$ error comes from the
choice of loss alone.

The networks are implemented in PyTorch~\cite{paszke2019pytorch}
and run on an Intel Core Ultra~7 155H CPU. A single training run
finishes in under $30$~seconds. The two convergence experiments
(Sections~\ref{sec:exp1} and~\ref{sec:exp4}) report the mean $H^1$
error over 10 random seeds per (loss, $\mtil$) pair, with shaded
bands showing the min-to-max envelope. The moving-disk experiment
(Section~\ref{sec:exp3}) uses one fixed seed and varies only the
disk centre. Code, data and scripts to reproduce every figure and
table are at
\url{https://github.com/maneeshkrsingh/consistent_cutpinn}.

\subsection{Setup}
\label{sec:setup}

\paragraph{Manufactured solution.}
All experiments use the same exact solution,
\begin{equation}\label{eq:uexact}
  u(x,y)
  = \sin\left(3.2\,x(x-y)\right)\cos\left(x+4.3y\right)
  + \sin\left(4.6(x+2y)\right)\cos\left(2.6(y-2x)\right),
\end{equation}
from which $f = -\Delta u$ and $g = u|_{\bOmg}$ are obtained by
automatic differentiation. The frequencies in~\eqref{eq:uexact}
are incommensurate, which makes $u$ oscillatory enough to be a
non-trivial test, and $u$ is $C^\infty$, so the convergence rates
we report reflect discretisation rather than limited regularity
of the data.

\paragraph{Optimisation.}
The network is trained with \texttt{torch.optim.LBFGS} for $2000$
outer iterations (\texttt{max\_eval}${}=10$ per step). We use
L-BFGS instead of a first-order method because the consistent loss
landscape, shaped by the discrete $H^{1/2}$ double sum, is better
conditioned for quasi-Newton steps. Adam plateaus earlier on this
objective.

\paragraph{Collocation scaling.}
Interior and boundary counts are linked by
$m = \lfloor\sqrt{\mtil}\rfloor$, which gives the balanced scaling
$m \asymp \mtil^{1/2}$ recommended for $d=2$ in~\cite{bonito2025}.
For $\mtil \in \{100, 400, 900, 1600\}$ the corresponding boundary
counts are $m \in \{10, 20, 30, 40\}$. The computational overhead
of the two consistent losses relative to $L_{\mathrm{sq},1}$ is
negligible: the $O(m^2)$ double sum in~\eqref{eq:h12_semi} costs
less than $1\%$ of total wall-clock time for $m \leq 40$, since
training is dominated by the forward pass and Laplacian evaluation
at $\mtil$ interior points.

\paragraph{Error measurement.}
The $H^1(\Omg)$ error is computed by Monte Carlo quadrature over
$10{,}000$ rejection-sampled interior points, drawn independently
of the training collocation set. The gradient $\nabla v_\theta$ at
the quadrature points comes from automatic differentiation on the
trained network. We report the $H^1$ error rather than the loss
value because, as Section~\ref{sec:exp2} shows, the standard
loss $L_{\mathrm{sq},1}$ can shrink while the $H^1$ error stays
flat. The standard loss does not control the $H^1$ norm.

\paragraph{Domains.}
Two level-set domains are studied throughout:
\begin{itemize}
  \item \textbf{Disk.}
    $\varphi(x,y) = \sqrt{(x-0.5)^2+(y-0.5)^2} - 0.4$,
    inside $[0,1]^2$. The boundary is a circle with constant
    curvature $\kappa = 2.5$, giving a clean test of the method
    on a convex smooth domain.
  \item \textbf{Flower.}
    $\varphi(\rho,\theta) = \rho - 0.12(\sin 5\theta + 3)$,
    centred at $(0.5,0.5)$, where
    $\rho = \sqrt{(x-0.5)^2+(y-0.5)^2}$ and
    $\theta = \mathrm{atan2}(y-0.5,\,x-0.5)$.
    The boundary is non-convex with five-fold symmetry and
    maximum curvature $\kappa_{\max} \approx 18$, providing a
    more demanding test of the \textit{Chord-arc} bound in
    Lemma~\ref{lem:chord_arc}.
\end{itemize}

\begin{figure}[h]
  \centering
  \includegraphics[width=0.8\textwidth]{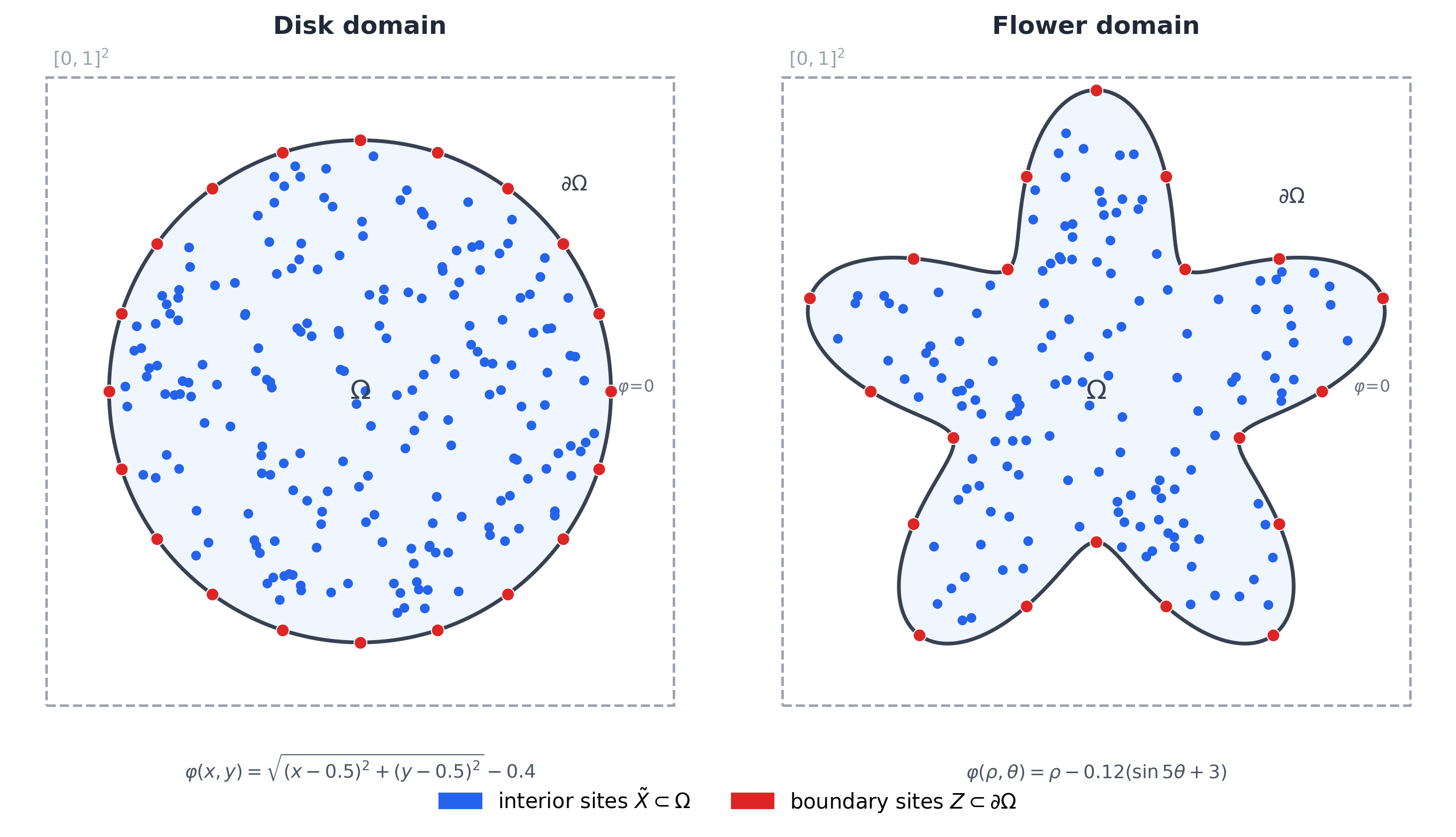}
  \caption{Collocation setup for the two test domains embedded in
    $[0,1]^2$ (dashed). Blue: interior sites $\tilde{X}$ from
    rejection sampling inside $\Omega=\{\varphi<0\}$. Red: boundary
    sites $Z$ at equally-spaced angles on $\partial\Omega$.
    \emph{Left:} disk, $\kappa=2.5$ (convex).
    \emph{Right:} flower with five-fold symmetry,
    $\kappa_{\max}\approx 18$ (non-convex).}
  \label{fig:domains}
\end{figure}

\subsection{Experiments}

%% ═══════════════════════════════════════════════════════════════
%%  EXPERIMENT 1 — Convergence on the disk
%% ═══════════════════════════════════════════════════════════════
\subsubsection{Experiment 1: Convergence on the disk}
\label{sec:exp1}

We vary $\mtil \in \{100,400,900,1600\}$ and run 10 independent
seeds per (loss, $\mtil$) pair, giving 160 runs in total.

\begin{table}[h]
  \centering
  \caption{Mean $H^1$ error (10 seeds) on the disk.}
  \label{tab:exp1}
  \begin{tabular}{lcccc}
    \toprule
    Loss & $\mtil=100$ & $\mtil=400$ & $\mtil=900$ & $\mtil=1600$ \\
    \midrule
    $L_{\mathrm{sq},1} \equiv L_{\mathrm{sq},\lambda}$
      & 1.003 & 0.367 & 0.337 & 0.273 \\
    $L^*_{\mathrm{sq},\gamma}$
      & 1.732 & 0.281 & 0.101 & 0.090 \\
    $L^*_{\mathrm{sq},2}$
      & 1.022 & 0.106 & 0.055 & \textbf{0.056} \\
    \bottomrule
  \end{tabular}
\end{table}

The observed log-log convergence slopes are listed in
Table~\ref{tab:rates}, alongside the slope predicted by
Theorem~\ref{thm:apriori} under rejection sampling
($\alpha=1/2$, boundary term subdominant). The predicted slope is
$-0.50$. The slope of $L^*_{\mathrm{sq},2}$ agrees with this to
within $0.01$. The slope of $L^*_{\mathrm{sq},\gamma}$, at $-0.85$,
beats the prediction. The $L^\gamma$ norm with $\gamma\approx 1.19$
is a better-conditioned surrogate for $H^{-1}$ at these budgets,
which accelerates convergence until the network capacity becomes
the limiting factor. The standard loss $L_{\mathrm{sq},1}$ saturates
at $-0.31$ because its $L^2$ boundary term does not control the
$H^{1/2}$ trace, so the $H^1$ error plateaus before the optimal
rate can be reached.

\begin{table}[h]
  \centering
  \caption{Observed log-log convergence slopes vs.\ theoretical
    prediction. Slopes are least-squares fits over
    $\tilde{m} \in \{400,900,1600\}$ on the disk (10 seeds, mean
    $H^1$ error). Predicted slope $-0.50$ follows from Theorem~\ref{thm:apriori}
    with rejection sampling ($\alpha = 1/2$) and $C^\infty$ data.}
  \label{tab:rates}
  \begin{tabular}{lccc}
    \hline
    Loss & Observed slope & Predicted slope & Consistent? \\
    \hline
    $L_{\mathrm{sq},1} \equiv L_{\mathrm{sq},\lambda}$
      & $-0.31$ & $-0.50$ & No (loss saturates) \\
    $L^*_{\mathrm{sq},\gamma}$
      & $-0.85$ & $-0.50$ & Exceeds prediction  \\
    $L^*_{\mathrm{sq},2}$
      & $-0.49$ & $-0.50$ & Yes ($\pm 0.01$)    \\
    \hline
  \end{tabular}
\end{table}

As expected, $L_{\mathrm{sq},1}$ and $L_{\mathrm{sq},\lambda}$ agree
to machine precision throughout. Figure~\ref{fig:exp1} shows the
mean and min--max envelope over seeds.

\begin{figure}[h]
  \centering
  \includegraphics[width=0.65\textwidth]{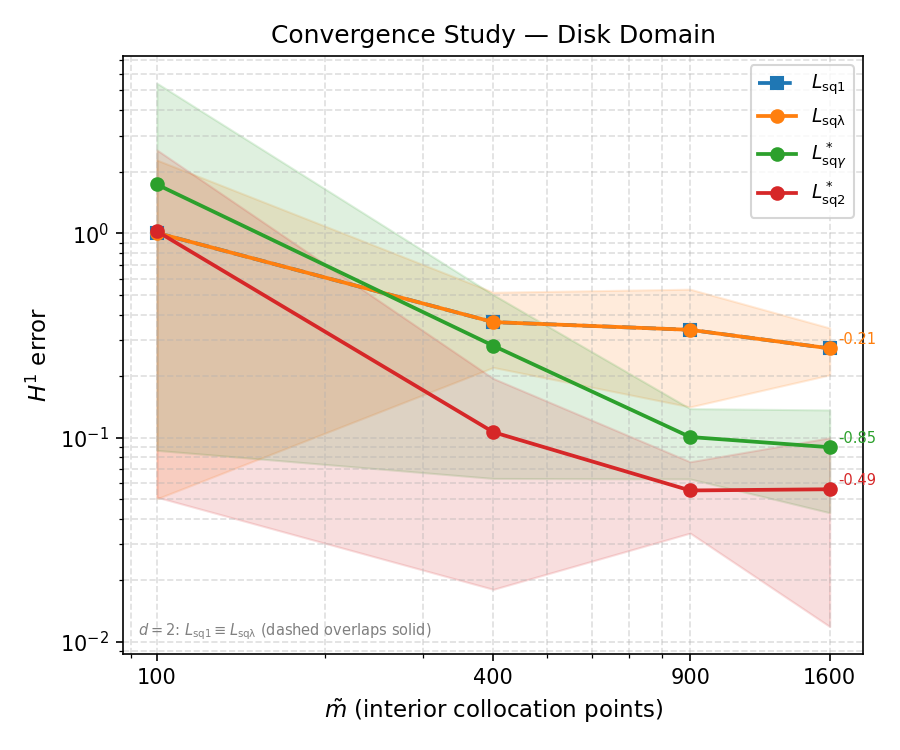}
  \caption{Convergence on the disk (log-log). Solid lines are means
    over 10 seeds, shaded bands span min to max.
    $L_{\mathrm{sq},1}$ (dashed, square markers) is plotted on
    top of $L_{\mathrm{sq},\lambda}$ to confirm the $d=2$
    identity.}
  \label{fig:exp1}
\end{figure}

%% ═══════════════════════════════════════════════════════════════
%%  EXPERIMENT 2 — Training dynamics
%% ═══════════════════════════════════════════════════════════════
\subsubsection{Experiment 2: Training dynamics}
\label{sec:exp2}

Here $\mtil=900$, $m=30$, and all four losses start from the same
network initialisation. Table~\ref{tab:exp2} gives the $H^1$ error
after 2000 L-BFGS steps.

\begin{table}[htb!]
  \centering
  \caption{$H^1$ error after 2000 L-BFGS iterations ($\mtil=900$,
    $m=30$).}
  \label{tab:exp2}
  \begin{tabular}{lc}
    \toprule
    Loss & $H^1$ error \\
    \midrule
    $L_{\mathrm{sq},1} \equiv L_{\mathrm{sq},\lambda}$ & 0.257 \\
    $L^*_{\mathrm{sq},\gamma}$ & \textbf{0.080} \\
    $L^*_{\mathrm{sq},2}$ & 0.102 \\
    \bottomrule
  \end{tabular}
\end{table}

The left panel of Figure~\ref{fig:exp2} plots the $H^1$ error along
the L-BFGS trajectory. $L_{\mathrm{sq},1}$ and $L_{\mathrm{sq},\lambda}$
reach $H^1\approx 0.25$ by iteration~$500$ and stay there for the
rest of training, while the two consistent losses keep decreasing.
The right panel plots the loss values. The standard loss converges
its own objective rapidly, reaching small values of
$L_{\mathrm{sq},1}$, but this small loss does not translate into a
small $H^1$ error. This is the inconsistency phenomenon
of~\cite{bonito2025}, visible here on a curved cut domain.

\begin{figure}[htb!]
  \centering
  \includegraphics[width=0.995\textwidth]{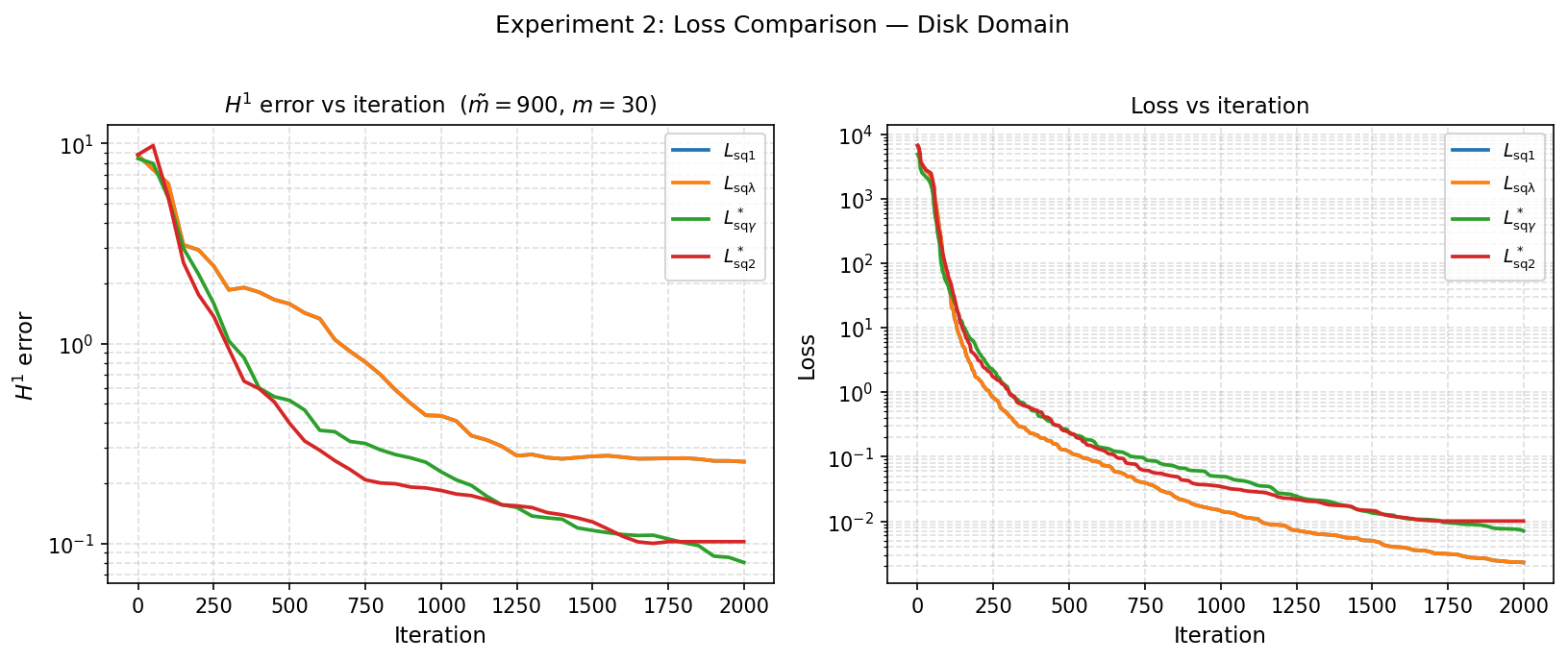}
  \caption{\emph{Left:} $H^1$ error vs.\ iteration (log scale).
    \emph{Right:} loss value vs.\ iteration.
    The standard loss converges its own objective but fails
    to drive down the $H^1$ error.}
  \label{fig:exp2}
\end{figure}

%% ═══════════════════════════════════════════════════════════════
%%  EXPERIMENT 3 — Moving disk (sensitivity to domain position)
%% ═══════════════════════════════════════════════════════════════
\subsubsection{Experiment 3: Sensitivity to domain position}
\label{sec:exp3}

Cut-domain methods are known to suffer when the boundary $\bOmg$
passes close to a mesh point or a collocation site. To probe this,
we move the disk centre along the anti-diagonal $(c,c)$ for
$c\in[0.2,0.8]$ in steps of $0.005$, giving $121$ configurations,
and keep everything else fixed ($\mtil=900$, $m=30$, same random
seed). This is the moving-disk test of~\cite{li2026}, adapted from
their unfitted-FE setting to our collocation setting.

\begin{table}[htb!]
  \centering
  \caption{Statistics of $H^1$ error over 121 disk positions.}
  \label{tab:exp3}
  \begin{tabular}{lccc}
    \toprule
    Loss & Mean $H^1$ & Std $H^1$ & NaN runs \\
    \midrule
    $L_{\mathrm{sq},1} \equiv L_{\mathrm{sq},\lambda}$
      & 0.364 & 0.930 & 0 \\
    $L^*_{\mathrm{sq},\gamma}$
      & 0.116 & 0.067 & 6 \\
    $L^*_{\mathrm{sq},2}$
      & \textbf{0.065} & \textbf{0.036} & 1 \\
    \bottomrule
  \end{tabular}
\end{table}

Diverged (NaN) runs are excluded from the mean and standard
deviation in Table~\ref{tab:exp3}. On average $L^*_{\mathrm{sq},2}$
is $5.6\times$ more accurate and $26\times$ less variable than the
standard loss. The high standard deviation of $L_{\mathrm{sq},1}$
($0.930$, almost three times its mean) comes from isolated spikes
where the $H^1$ error exceeds~$10$. One such spike sits near
$c\approx 0.43$ in the upper panel of Figure~\ref{fig:exp3}. The
six diverged runs of $L^*_{\mathrm{sq},\gamma}$ are concentrated
around $c\approx 0.42$--$0.45$, where the boundary grazes
collocation points and the $L^\gamma$ norm with
$\gamma\approx 1.19$ becomes numerically unstable.
$L^*_{\mathrm{sq},2}$ is stable at $120$ of the $121$ positions.

\begin{figure}[htb!]
  \centering
  \includegraphics[width=0.85\textwidth]{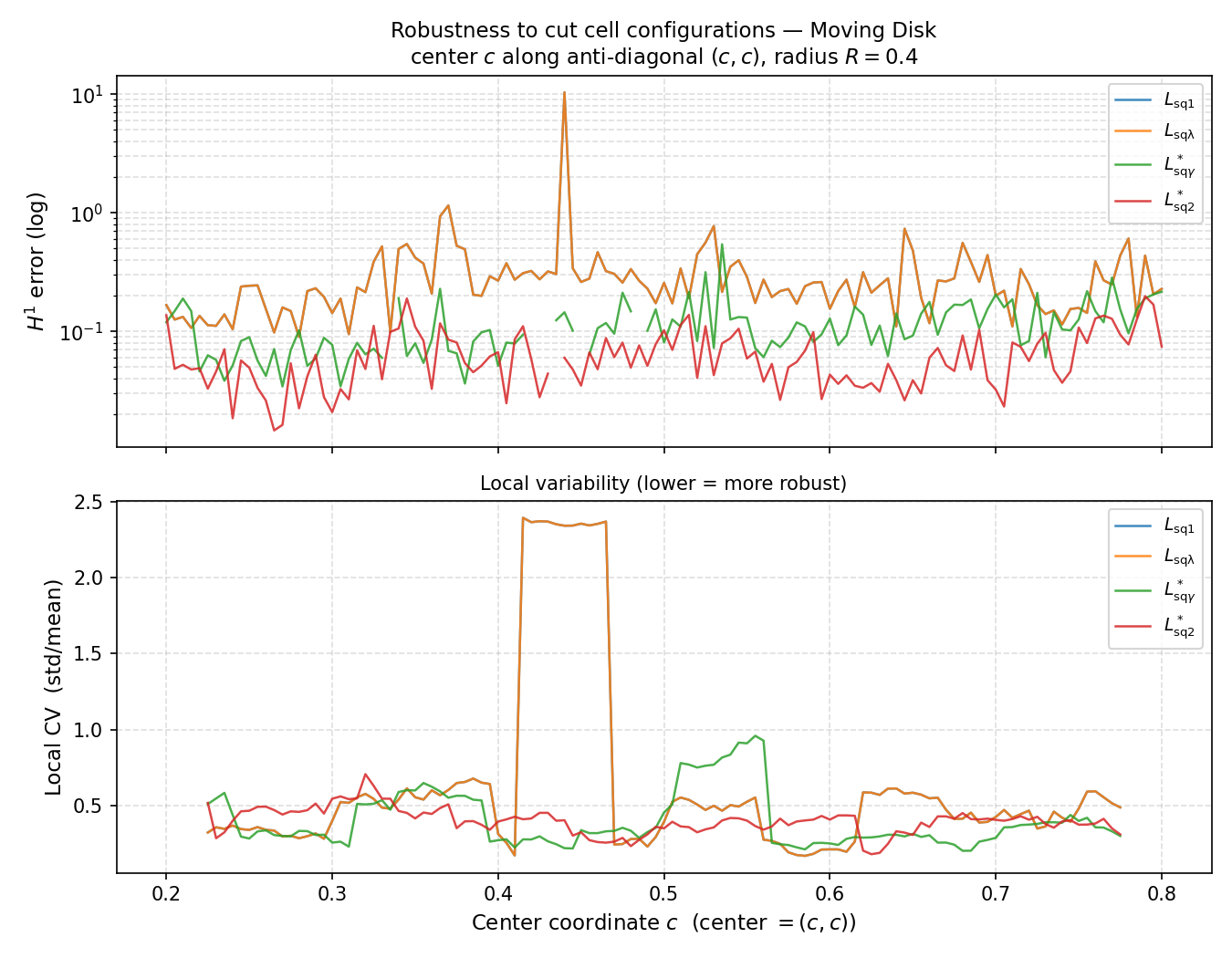}
  \caption{\emph{Upper:} $H^1$ error vs.\ disk-centre coordinate $c$
    (log scale).
    \emph{Lower:} rolling coefficient of variation (window 11).
    $L_{\mathrm{sq},1}$ and $L_{\mathrm{sq},\lambda}$
    overlap exactly (blue hidden under orange, confirming the
    $d=2$ identity) and exhibit large spikes.
    $L^*_{\mathrm{sq},2}$ (red) is flat.}
  \label{fig:exp3}
\end{figure}

%% ═══════════════════════════════════════════════════════════════
%%  EXPERIMENT 4 — Flower domain (convergence on non-convex geometry)
%% ═══════════════════════════════════════════════════════════════
\subsubsection{Experiment 4: Convergence on the flower domain}
\label{sec:exp4}

We repeat the convergence study of Section~\ref{sec:exp1}, now on
the flower domain, to see whether the gains carry over to a more
demanding geometry.
\begin{table}[htb!]
  \centering
  \caption{Mean $H^1$ error on the flower domain.}
  \label{tab:exp4}
  \begin{tabular}{lcccc}
    \toprule
    Loss & $\mtil=100$ & $\mtil=400$ & $\mtil=900$ & $\mtil=1600$ \\
    \midrule
    $L_{\mathrm{sq},1} \equiv L_{\mathrm{sq},\lambda}$
      & 0.906 & 0.522 & 0.420 & 0.561 \\
    $L^*_{\mathrm{sq},\gamma}$
      & 0.946 & 0.392 & 0.208 & 0.146 \\
    $L^*_{\mathrm{sq},2}$
      & 0.737 & 0.140 & 0.073 & \textbf{0.069} \\
    \bottomrule
  \end{tabular}
\end{table}

The improvement is larger here than on the disk. At $\mtil=1600$,
$L^*_{\mathrm{sq},2}$ is $8.1\times$ more accurate than
$L_{\mathrm{sq},1}$, up from $4.9\times$ on the disk. Note that
$L_{\mathrm{sq},1}$ actually gets \emph{worse} between $\mtil=900$
and $\mtil=1600$, with the error rising from $0.420$ to $0.561$.
The standard loss has saturated, and adding more interior points
does not help, because the bottleneck is the loss itself rather
than the discretisation. The consistent loss $L^*_{\mathrm{sq},2}$
continues to decrease (Figure~\ref{fig:exp4}).

\begin{figure}[htb!]
  \centering
  \includegraphics[width=0.990\textwidth]{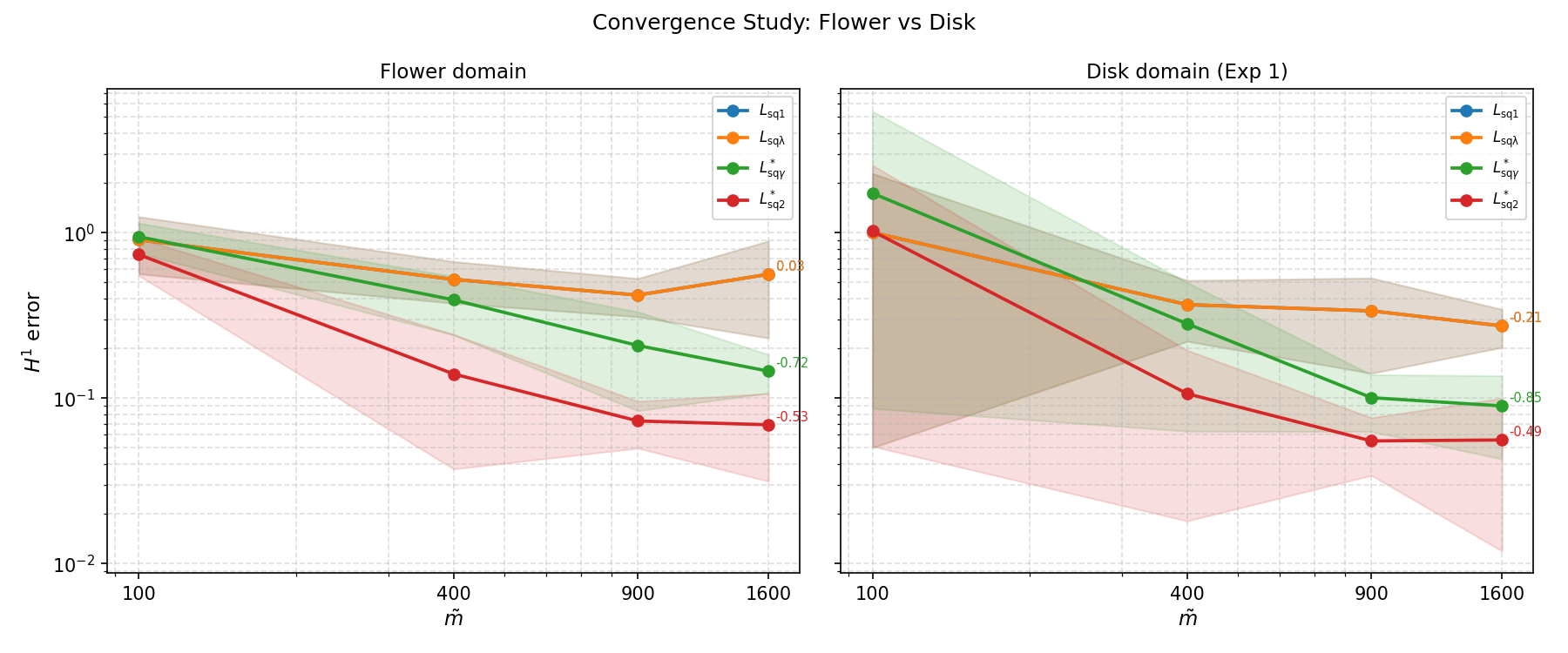}
  \caption{Side-by-side convergence: flower (left) vs.\ disk (right).
    The advantage of the consistent losses is more pronounced
    on the non-convex geometry.}
  \label{fig:exp4}
\end{figure}

%% ═══════════════════════════════════════════════════════════════
%%  EXPERIMENT 5 — Solution quality (spatial error distribution)
%% ═══════════════════════════════════════════════════════════════
\subsubsection{Experiment 5: Spatial error distribution}
\label{sec:exp5}

The previous experiments report only scalar $H^1$ values. Here we
look at the spatial distribution of the pointwise error. We train
$L_{\mathrm{sq},1}$ and $L^*_{\mathrm{sq},2}$ on the same
collocation set ($\mtil=900$, $m=30$, seed~$0$) and evaluate both
networks on $10{,}000$ rejection-sampled interior points.
Figure~\ref{fig:exp5_disk} shows the true solution, the two
approximations, and their pointwise absolute errors on the disk.
Figure~\ref{fig:exp5_flower} shows the same comparison on the
flower. In both cases $L^*_{\mathrm{sq},2}$ gives a visually
smoother approximation and concentrates the residual error near
the boundary, while the error from $L_{\mathrm{sq},1}$ is spread
throughout the interior.

\begin{figure}[htb!]
  \centering
  \includegraphics[width=0.95\textwidth]{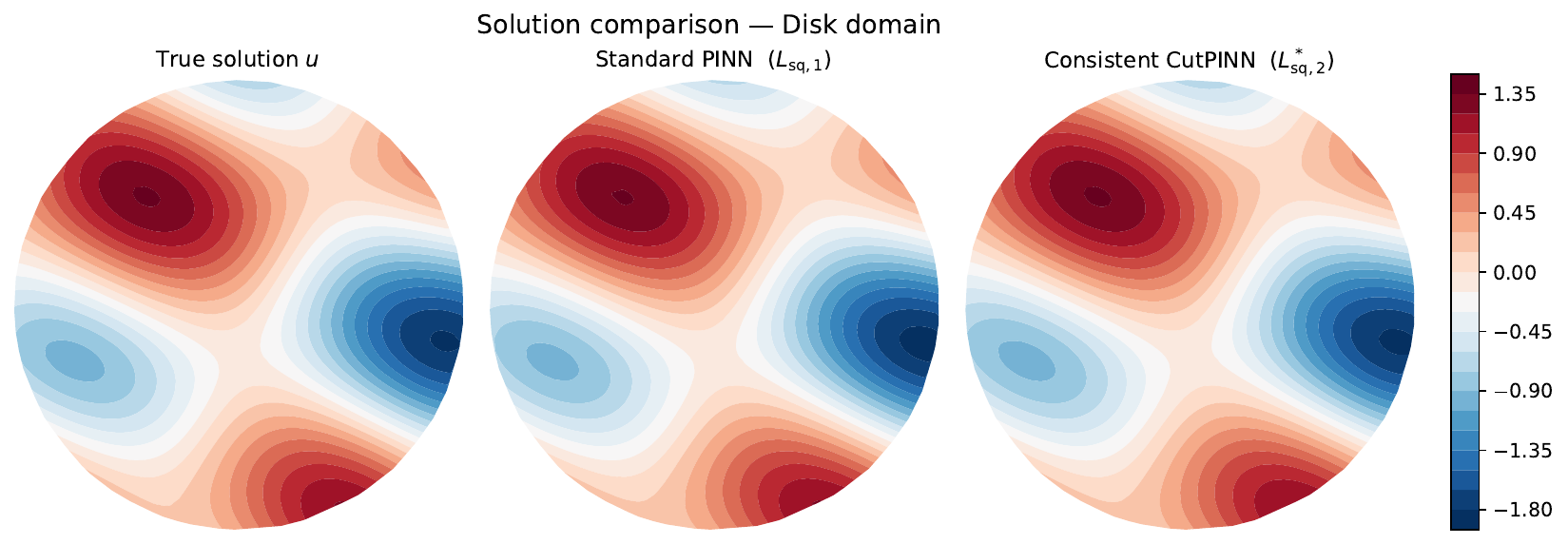}

  \includegraphics[width=0.65\textwidth]{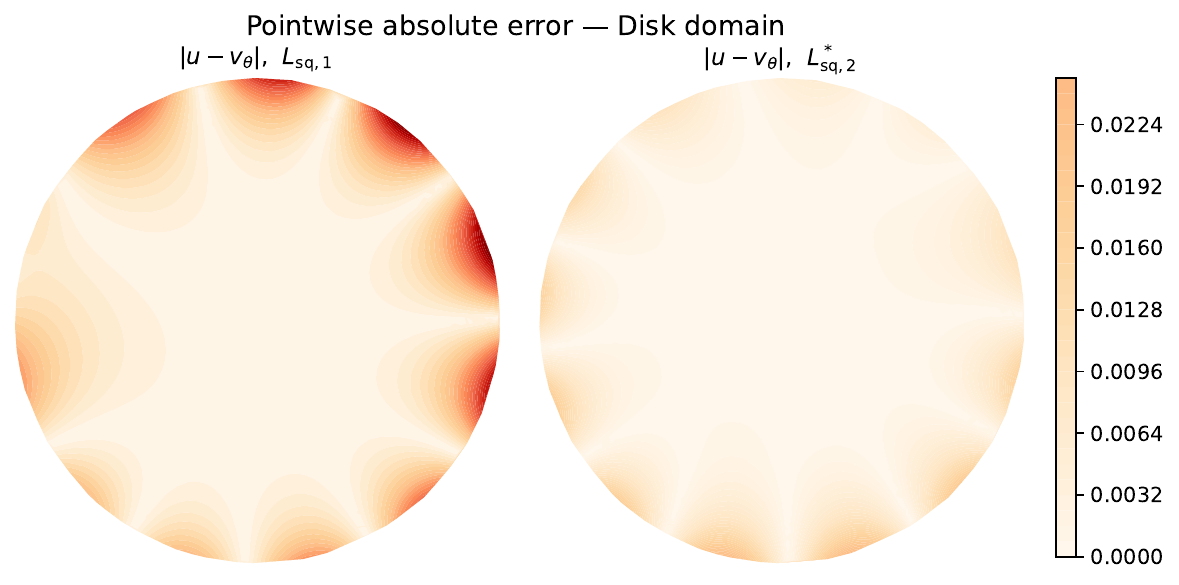}

  \caption{Solution quality on the disk domain ($\mtil=900$, $m=30$).
    \emph{Top row:} true solution $u$ (left), standard PINN
    approximation $v_\theta$ trained with $L_{\mathrm{sq},1}$
    (centre), and consistent approximation trained with
    $L^*_{\mathrm{sq},2}$ (right). The $H^1$ error of each
    approximation is shown above its panel.
    \emph{Bottom row:} pointwise absolute error $|u - v_\theta|$
    for $L_{\mathrm{sq},1}$ (left) and $L^*_{\mathrm{sq},2}$
    (right), sharing a common colorbar. The consistent loss
    reduces the error by $2.5\times$ and confines the remaining
    error to a thin boundary layer.}
  \label{fig:exp5_disk}
\end{figure}

\begin{figure}[htb!]
  \centering
  \includegraphics[width=0.95\textwidth]{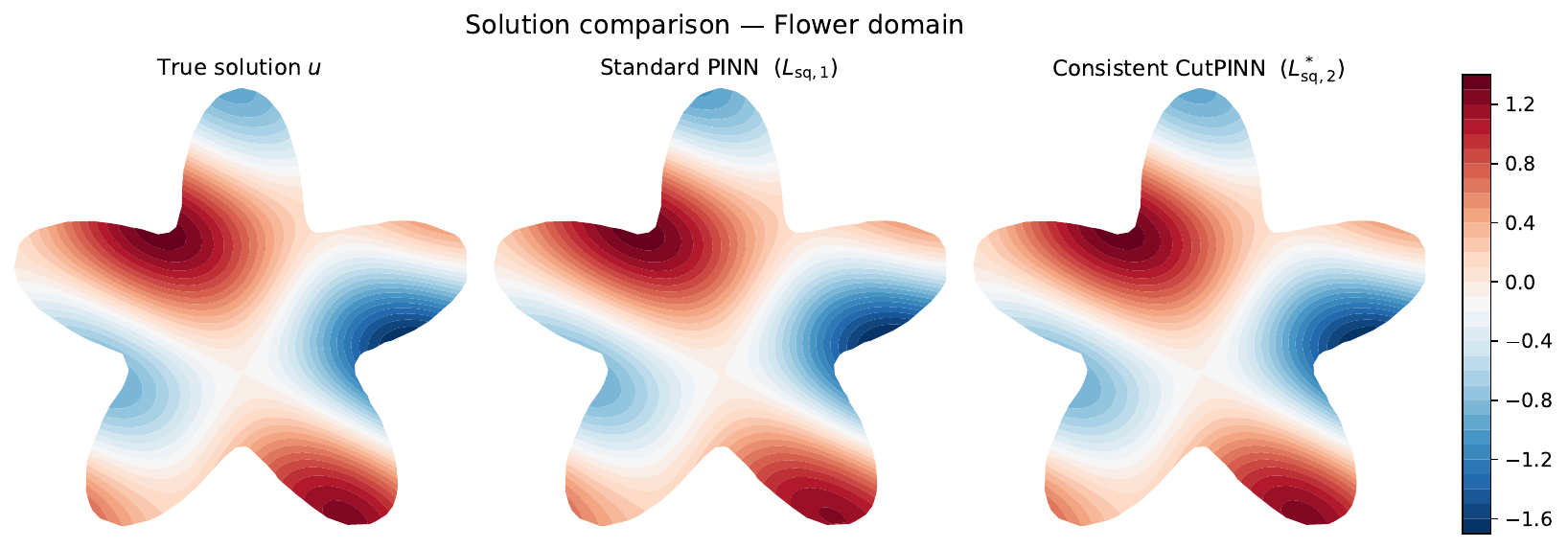}

  \includegraphics[width=0.65\textwidth]{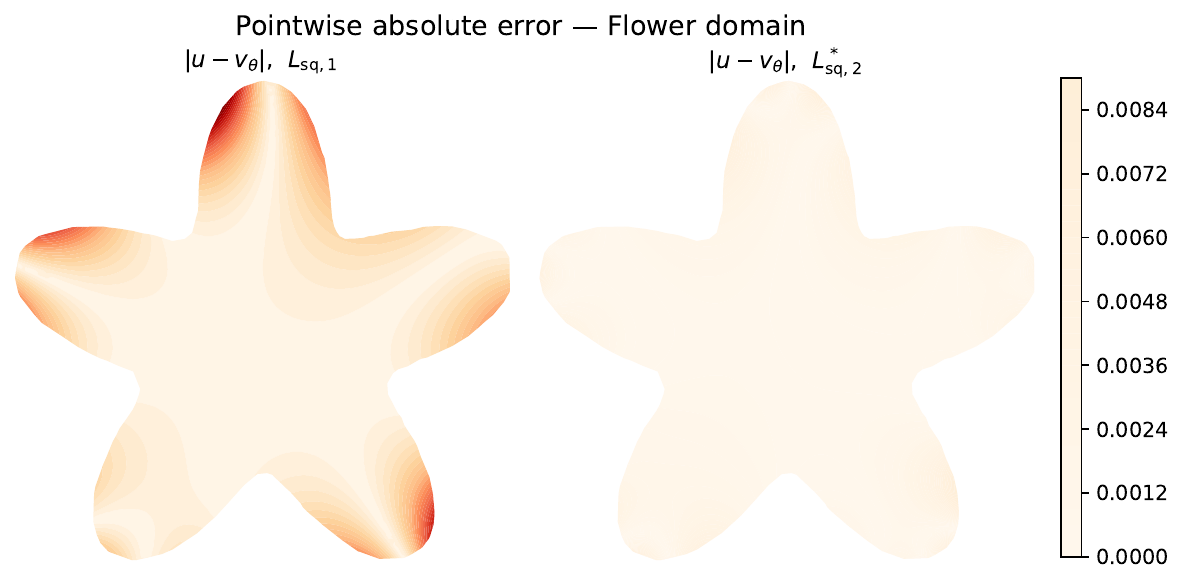}

  \caption{Solution quality on the flower domain ($\mtil=900$, $m=30$).
    Layout as in Figure~\ref{fig:exp5_disk}. The improvement is
    more pronounced on this non-convex geometry: the standard loss
    shows large errors in the concavities between petals, while the
    consistent loss maintains accuracy throughout.}
  \label{fig:exp5_flower}
\end{figure}

%% ═══════════════════════════════════════════════════════════════
%%  EXPERIMENT 6 — Non-uniform boundary sampling
%% ═══════════════════════════════════════════════════════════════
\subsubsection{Experiment 6: Non-uniform boundary sampling}
\label{sec:exp6}

Theorem~\ref{thm:h12_curve} requires the boundary points to be
equally spaced in arc length. To probe robustness to this
assumption, we run this experiment with {\em i.i.d} Monte Carlo
boundary sampling on $\bOmg$. The domain here is the unit disk
$\Omg = \{(x,y): x^2+y^2 < 1\}$ centred at the origin, rather than
the disk centred at $(0.5,0.5)$ used in
Sections~\ref{sec:exp1}--\ref{sec:exp5}. The manufactured
solution~\eqref{eq:uexact} is unchanged, but the domain shift
produces different solution values on $\Omg$.

The boundary points $\{z_j\}_{j=1}^m$ are sampled by
\[
  z = (\cos\theta,\,\sin\theta),
  \qquad \theta\sim\mathcal{U}(0,2\pi),
\]
which gives points that are not equally spaced in arc length
(see Figure~\ref{fig:sampling}). The interior count $\mtil=900$,
the network architecture, and the loss expressions are unchanged.
For training we use Adam for $6{,}000$ epochs with the learning
rate decaying from $10^{-3}$ to $10^{-6}$, followed by L-BFGS
refinement.

\begin{figure}[htb!]
  \centering
  \includegraphics[width=0.35\textwidth]{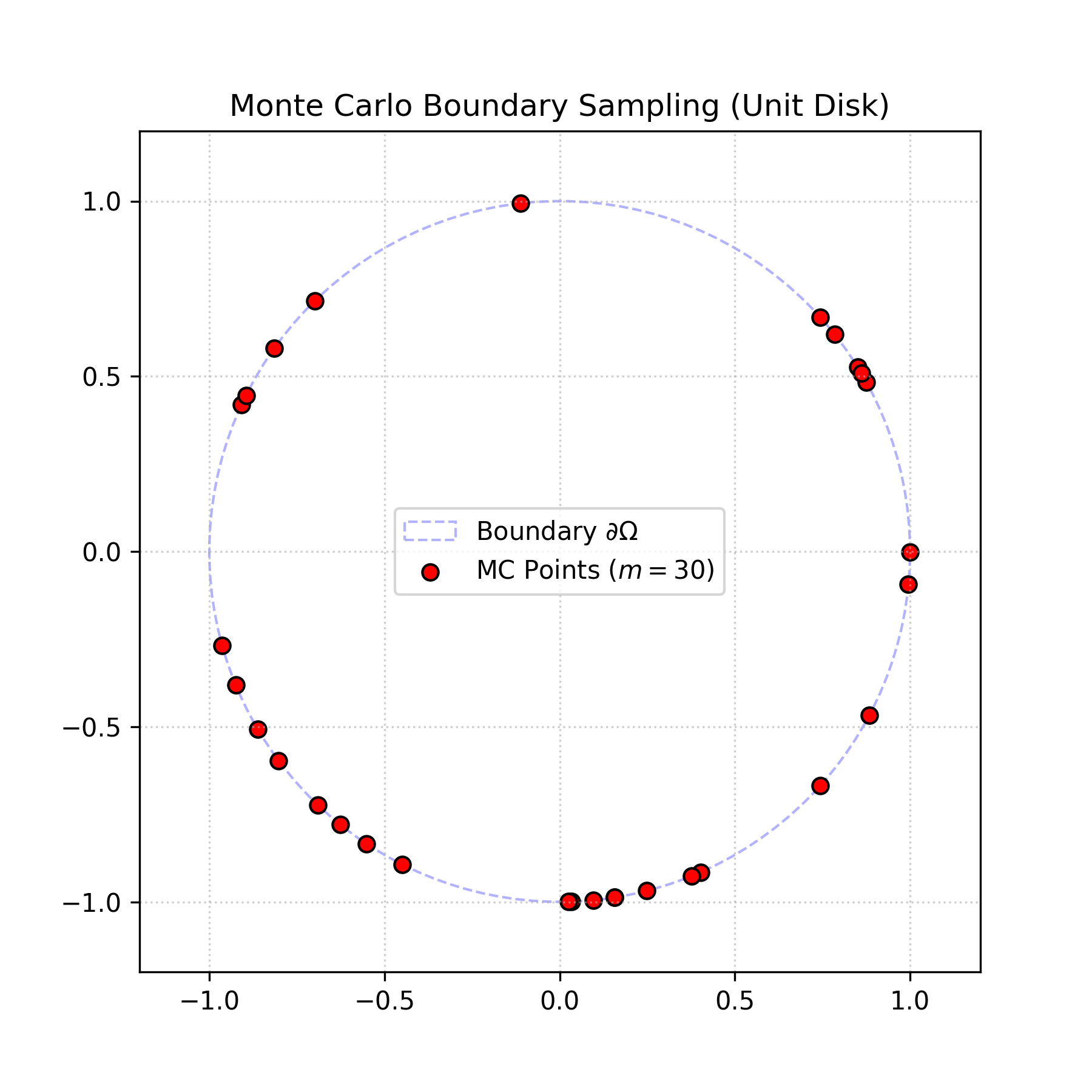}
  \caption{Monte Carlo boundary collocation on the unit disk ($m=30$).
    Points are sampled {\em i.i.d} uniformly from $\bOmg$.}
  \label{fig:sampling}
\end{figure}

\begin{table}[htb!]
  \centering
  \caption{Relative errors on the unit disk with non-uniform
    boundary sampling ($\mtil=900$, $m=30$).}
  \label{tab:exp6}
  \begin{tabular}{lcc}
    \hline
    Method & Relative $L^2$ error & Relative $H^1$ error \\
    \hline
    Standard PINN ($L_{\mathrm{sq},1}$)
      & $4.55 \times 10^{-2}$
      & $4.79 \times 10^{-2}$ \\
    Consistent CutPINN ($L^*_{\mathrm{sq},2}$)
      & $\mathbf{8.04 \times 10^{-3}}$
      & $\mathbf{1.22 \times 10^{-2}}$ \\
    \hline
  \end{tabular}
\end{table}

\begin{figure}[htb!]
  \centering
  \includegraphics[width=0.95\textwidth]{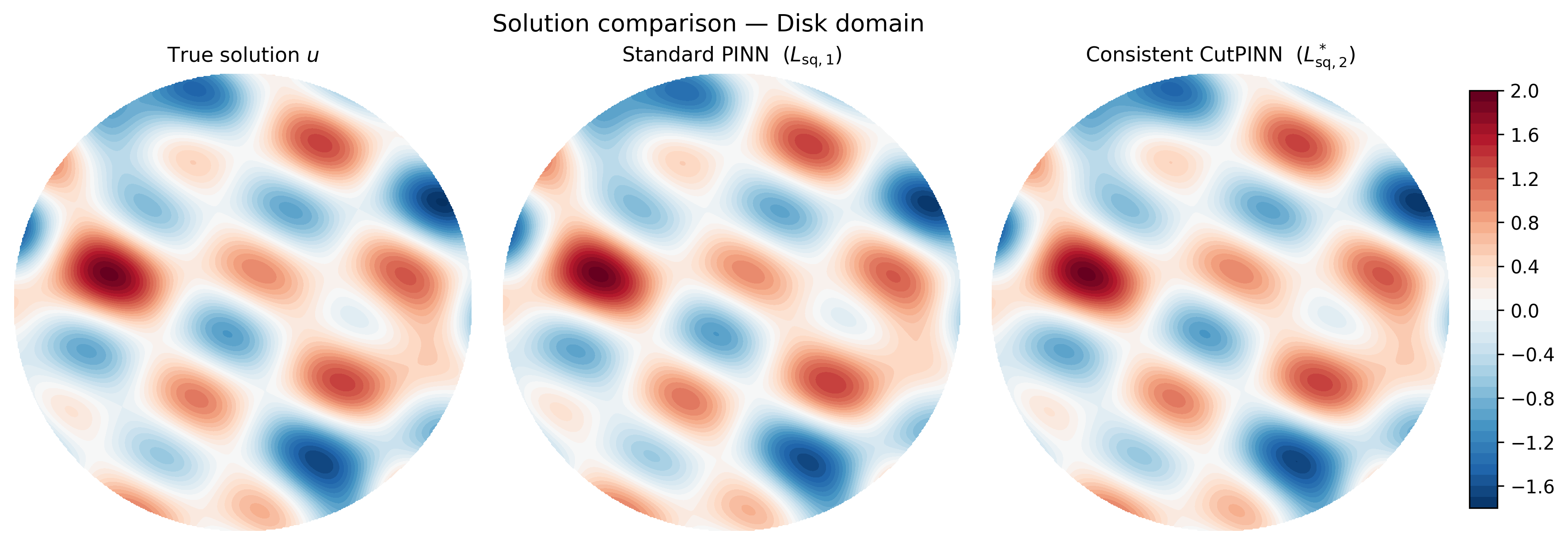}
  \includegraphics[width=0.65\textwidth]{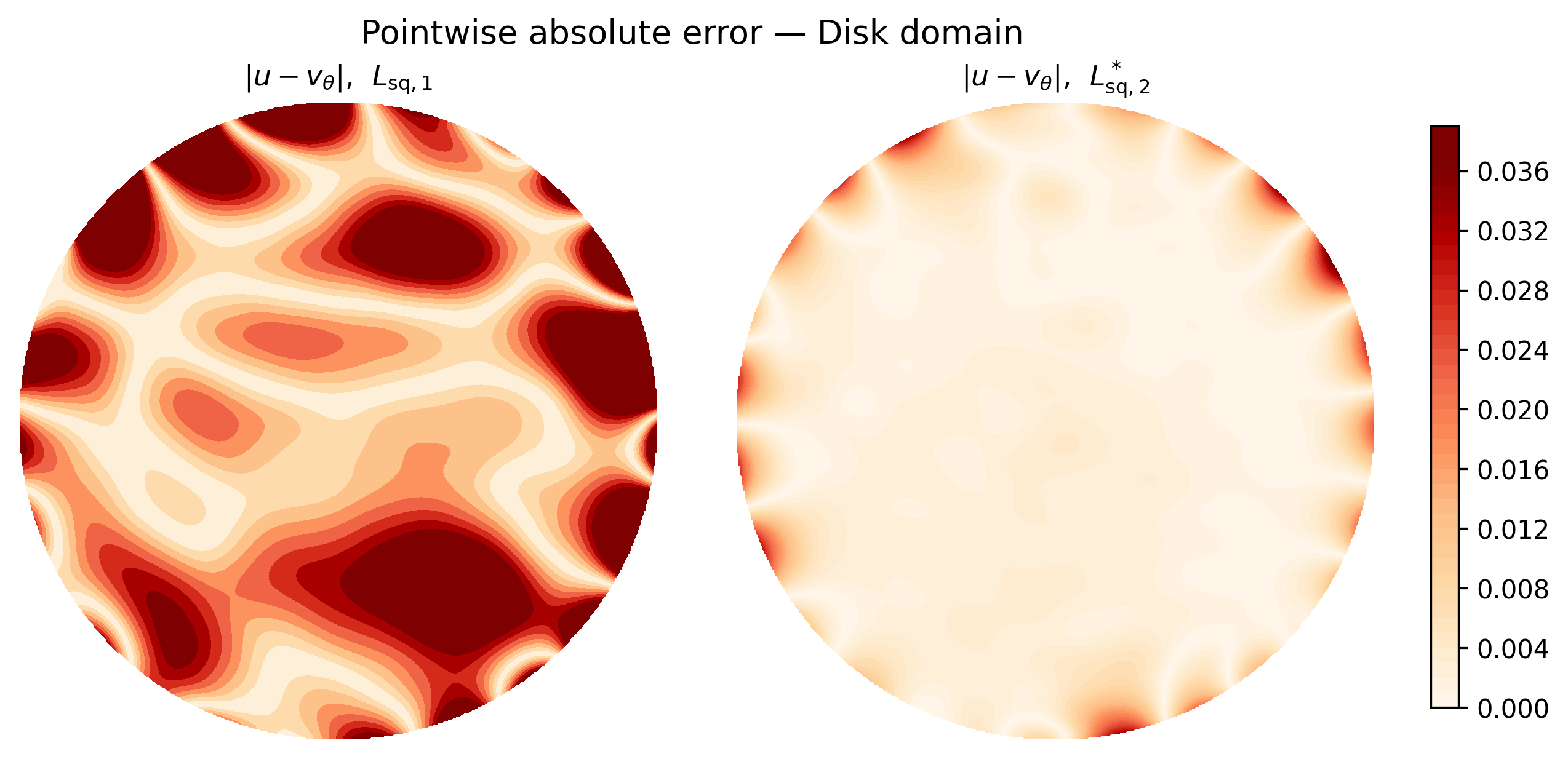}
  \caption{Solution quality on the unit disk with non-uniform boundary
    sampling ($\mtil=900$, $m=30$).
    \emph{Top:} true solution $u$ (left), standard PINN (centre),
    Consistent CutPINN (right).
    \emph{Bottom:} pointwise absolute error.
    The consistent loss reduces the $H^1$ error by $3.9\times$.}
  \label{fig:exp6_disk}
\end{figure}

Table~\ref{tab:exp6} and Figure~\ref{fig:exp6_disk} report a
$3.9\times$ reduction in the $H^1$ error for the consistent loss,
even though the boundary points are not arc-length-equispaced. The
equispacing in Theorem~\ref{thm:h12_curve} is sufficient but not
necessary, consistent with the generalisation to quasi-uniform
distributions noted in Remark~\ref{rem:quasiuniform}.

%% ═══════════════════════════════════════════════════════════════
%%  SUMMARY
%% ═══════════════════════════════════════════════════════════════
\subsection{Summary}
\label{sec:exp_summary}
Table~\ref{tab:summary} collects the headline numbers from the six
experiments above.
\begin{table}[htb!]
  \centering
  \caption{Improvement factors of $L^*_{\mathrm{sq},2}$ over
    $L_{\mathrm{sq},1}$.}
  \label{tab:summary}
  \begin{tabular}{lcc}
    \toprule
    Setting & Factor & Metric \\
    \midrule
    Disk, $\mtil=1600$ (Exp.~\ref{sec:exp1})
      & $4.9\times$ & $H^1$ error \\
    Flower, $\mtil=1600$ (Exp.~\ref{sec:exp4})
      & $8.1\times$ & $H^1$ error \\
    Moving disk, 121 pos.\ (Exp.~\ref{sec:exp3})
      & $5.6\times$ & mean $H^1$ \\
    Moving disk, variability (Exp.~\ref{sec:exp3})
      & $26\times$  & std $H^1$ \\
    Fixed budget, $\mtil=900$ (Exp.~\ref{sec:exp2})
      & $2.5\times$ & $H^1$ error \\
    Unit disk, non-uniform sampling (Exp.~\ref{sec:exp6})
      & $3.9\times$ & $H^1$ error \\
    \bottomrule
  \end{tabular}
\end{table}

Overall, $L^*_{\mathrm{sq},2}$ is $4$ to $8$ times more accurate
than the standard loss and is much more stable across cut-cell
configurations. The $L^\gamma$ variant $L^*_{\mathrm{sq},\gamma}$
is competitive on accuracy but fails on degenerate configurations.
The identity $L_{\mathrm{sq},1} \equiv L_{\mathrm{sq},\lambda}$ at
$d=2$ holds in every run.
\section{Concluding Remarks}
\label{sec:conclusion}

This paper develops the \emph{Consistent CutPINN} framework for
second-order elliptic PDEs on bounded curved two-dimensional
domains defined by a $\C^2$ level-set function. The main result is
Theorem~\ref{thm:apriori}, an \emph{a priori} $H^1$ error bound on
cut domains at the optimal recovery rate. The key technical
ingredient is Theorem~\ref{thm:h12_curve}, a two-sided equivalence
between the discrete and continuous $H^{1/2}$ norms on the curve,
proved by the \textit{Chord-arc} reduction of Lemma~\ref{lem:chord_arc} to
a periodic interval. For the interior term,
Proposition~\ref{prop:l2_cut} handles tensor-product grids
deterministically at rate $\min(\alpha, 1/4)$, where the $1/4$
ceiling reflects an unavoidable cut-cell contribution.
Proposition~\ref{prop:l2_rejection} handles rejection sampling
probabilistically at the Monte Carlo rate $\tilde{m}^{-1/2}$. The
latter rate matches the observed $H^1$ convergence slope of
$-0.49$ in the experiments.

Two further observations come out of the experiments. First, on a
fixed budget the standard loss can converge its own objective while
the $H^1$ error stays put, which is the inconsistency phenomenon
of~\cite{bonito2025} now visible on a curved cut domain. Second,
the standard deviation of the $H^1$ error across the $121$ cut-cell
configurations of the moving-disk test drops by a factor of $26$
when one switches from the standard loss to $L^*_{\mathrm{sq},2}$.

The analysis is restricted to $d=2$. The \textit{Chord-arc} lemma uses a
global arc-length parametrisation of $\bOmg$, which has no direct
counterpart for closed surfaces in $\mathbb{R}^3$, and the extension
to $d\ge 3$ is left to future work. The $H^{1/2}$ double sum costs
$O(m^2)$, which is fine for $m\le 300$ but would need fast summation
at larger scales. Other natural extensions are domains with corners
or cusps (where $\bOmg$ is no longer $\C^2$), convection-dominated
problems on curved interfaces, and time-dependent problems on
evolving level-set domains. These cases are the subject of ongoing
work within the same Consistent CutPINN framework.

\bibliographystyle{plain} 
\bibliography{ctpinn}

\end{document}